\documentclass[12pt]{amsart}

\usepackage{dsfont}
\usepackage{tikz}
\usepackage{amssymb}
\usepackage{amsmath}
\usepackage{enumerate}

\DeclareMathOperator\One{\mathds 1}
\DeclareMathOperator\Ne{\mathcal N}
\DeclareMathOperator\A{\mathcal A}
\DeclareMathOperator\R{\mathbb R}
\DeclareMathOperator\E{\mathcal E}
\DeclareMathOperator\EE{\mathbb E}
\DeclareMathOperator\C{\mathbb C}
\DeclareMathOperator\Z{\mathbb Z}
\DeclareMathOperator\N{\mathbb N}
\DeclareMathOperator\Q{\mathbb Q}
\DeclareMathOperator\T{\mathbb T}
\DeclareMathOperator\csm{c^{sm}}
\DeclareMathOperator\mC{mC}
\DeclareMathOperator\Rep{Rep}
\DeclareMathOperator\HH{\mathcal H}
\DeclareMathOperator\KK{\mathcal K}
\DeclareMathOperator\SSS{\mathcal S}
\DeclareMathOperator\fac{F}
\DeclareMathOperator\G{\mathcal G}
\DeclareMathOperator\Hom{\mathrm Hom}
\DeclareMathOperator\Inj{\mathrm Inj}
\DeclareMathOperator\Ext{\mathrm Ext}

\DeclareMathOperator\rk{rk}
\DeclareMathOperator{\GL}{GL}
\DeclareMathOperator{\BGL}{BGL}
\DeclareMathOperator\Fl{\mathcal F}
\DeclareMathOperator\CQ{\C\! Q}
\DeclareMathOperator\ep{\varepsilon}
\DeclareMathOperator\OO{\mathcal O}
\DeclareMathOperator\V{\mathcal V}
\DeclareMathOperator\Y{\mathcal Y}
\DeclareMathOperator\Exp{Exp^*}
\DeclareMathOperator\Expy{Exp_y^*}
\newcommand{\ggamma}{{\boldsymbol \gamma}}
\newcommand{\ddelta}{{\boldsymbol \delta}}
\newcommand{\zv}{{\boldsymbol 0}}
\DeclareMathOperator\codim{codim}

\newtheorem{fact}{Fact}[section]
\newtheorem{lemma}[fact]{Lemma}
\newtheorem{theorem}[fact]{Theorem}
\newtheorem{definition}[fact]{Definition}
\newtheorem{example}[fact]{Example}
\newtheorem{rremark}[fact]{Remark}
\newenvironment{remark}{\begin{rremark} \rm}{\end{rremark}}
\newtheorem{proposition}[fact]{Proposition}
\newtheorem{corollary}[fact]{Corollary}
\newtheorem{conjecture}[fact]{Conjecture}

\author{R. Rim\'anyi}
\address{Department of Mathematics, University of North Carolina at Chapel Hill, USA}
\email{rimanyi@email.unc.edu}

\title{Motivic characteristic classes in cohomological Hall algebras}

\begin{document}

\begin{abstract}
The equivariant Chern-Schwartz-MacPherson (CSM) class and the equivariant Motivic Chern (MC) class are important characteristic classes of singular varieties in cohomology and K theory---and their theory overlaps with the theory of Okounkov's stable envelopes. We study CSM and MC classes for the orbits of Dynkin quiver representations. We show that the problem of computing the CSM and MC classes of all these orbits can be reduced to some basic classes $c^o_\beta$, $C^o_\beta$ parameterized by positive roots $\beta$. We prove an identity in a deformed version of Kontsevich-Soibelman's Cohomological (and K-theoretic) Hall Algebra (CoHA, KHA), namely, that a product of exponentials of $c_\beta^o$ (or $C_\beta^o$) classes formally depending on a stability function $Z$, does {\em not} depend on $Z$. This identity---which encodes infinitely many identities among rational functions in growing number of variables---has the structure of Donaldson-Thomas type quantum dilogarithm identities. Using a wall-crossing argument we present the $c_\beta^o$, $C_\beta^o$ classes as certain commutators in the CoHA, KHA.
\end{abstract}

\maketitle


\section{Introduction}

In this paper we discuss the relationship between cohomological and K-theoretic motivic characteristic classes of Dynkin quiver orbits, Hall algebras, and Donaldson-Thomas type identities. 

\subsection{Motivic characteristic classes}

Characteristic cohomology classes of singular varieties is a powerful tool in studying the geometry of the varieties. The most obvious characteristic class is the {\em fundamental class}. Following up on a conjecture of Deligne and Grothendieck, in \cite{M} MacPherson defined an inhomogeneous deformation of the fundamental class---which we will call the Chern-Schwartz-MacPherson (CSM) class. The push-forward and pull-back properties of the CSM class are similar to those of the fundamental class, and for smooth varieties the CSM class recovers the total Chern class of the variety. Most interestingly, the CSM class can be defined for locally closed subvarieties too, and it is additive---or ``motivic''. 

In \cite{BSY} Brasselet, Sch{\"u}rmann, and Yokura found the K theory analogue of the CSM class and named it motivic Chern (MC) class. Its behavior with respect to push-forward and pull-back is similar to that of the CSM class.
For a smooth variety the motivic Chern class is Hirzebruch's $\lambda_y$-class of the cotangent bundle, which can be regarded as
the K-theoretic total Chern class. For a smooth compact variety the integral of the motivic Chern class is the $\chi_y$-genus of the variety. By definition the MC class is motivic too.

Often the varieties whose characteristic classes we are studying can be obtained as degeneracy loci varieties. Standard arguments in degeneracy loci theory reduce the problem of finding charactersitic classes of degeneracy loci to finding {\em equivariant} characteristic classes of invariant subvarieties of representations (cf. Thom polynomials of singularities, or quiver polynomials). The equivariant versions of CSM and MC classes exist, see \cite{O1, W, O2,  FRW}. It is shown in \cite{RV,FRcsm, AMSS, FRW, AMSS2} that for certain varieties the CSM and MC classes are the same as the {\em cohomological and K-theoretic stable envelope} classes introduced by Okounkov and his co-authors, see \cite{MO1, O, AO1} and references therein, as well as \cite{GRTV, RTV3}. 

\subsection{Equivariant CSM and MC classes of orbits of quiver representations}
Representations associated with Dynkin quivers have finitely many orbits. Moreover, Gabriel's theorem and the notion of Aulander-Reiten quiver reduces lots of geometric questions about the orbits to simpler algebraic or combinatorial questions. Using these reductions the equivariant fundamental classes of orbit closures are now quite well understood (often under the name of quiver polynomials), see e.g. \cite{B, qr} and references therein---although some key positivity conjectures are still open.

The {\bf first achievement} of this paper is a formula for the equivariant CSM and MC classes of orbits of Dynkin quiver representations: the results are Theorem \ref{thm:1} and \ref{thm:2} as well as their interplay explained in Section \ref{sec:interplay}. The main tool in the proof is a modification of Reineke's resolution \cite{Re}. For the $A_2$ quiver the CSM and MC classes of orbits were already studied in \cite{PP,FRcsm,FRW}, but even in this case our formulas are new. 

\subsection{Cohomological and K-theoretic Hall algebras}
Inspired by string theory, Kontsevich and Soibelman introduced the notion of Cohomological Hall Algebras (CoHA) in \cite{KS}. For quivers this algebra is built on a vector space which is the direct sum of equivariant cohomology rings of quiver representations for different dimension vectors. Hence the characteristic classes of quiver orbits we mentioned above are elements of the CoHA. It is shown in \cite{rrcoha} that the fundamental classes of Dynkin quiver orbit closures, as elements of the CoHA, can be obtained as $1*1*\ldots*1$. Here $*$ is the non-commutative multiplication of the CoHA and the different 1's are some special elementary classes. 

In \cite{YZ} a deformed version of the CoHA is defined and it is also extended from cohomology to other extraordinary cohomology theories. In this paper by CoHA and KHA we will mean this deformed version of \cite{YZ} for cohomology and K theory.

In Section \ref{sec:co} we define a new cohomological characteristic class $c^o_\beta$ and a new K-theoretic characteristic class $C^o_\beta$ for every positive root $\beta$ for every Dynkin quiver. The {\bf second achievement} of the paper is Theorem \ref{thm:main1}, in which we present the CSM and MC classes of Dynkin quiver orbits (up to a factorial or quantum factorial) in the form $c^o_{\beta_1}*c^o_{\beta_2}*\ldots*c^o_{\beta_N}$ and $C^o_{\beta_1}*C^o_{\beta_2}*\ldots*C^o_{\beta_N}$. In type $A$ we have $c^o_\beta=1$ for every $\beta$, hence the formula is $1*\ldots *1$ just like for the fundamental class, but this coincidence only holds in type $A$. 

The usual features of characteristic classes are positivity and stability. Positivity means that the classes, expanded in some suitably chosen basis, have non-negative or alternating sign coefficients. Stability means that the characteristic class of a ``simpler'' orbit should be computable from the characteristic class of a ``more difficult'' orbit. The above presentation of CSM and MC classes as products of $c^o_\beta$ or $C^o_\beta$ classes is a manifestation of stability. Positivity questions are not studied in this paper. 

\subsection{Donaldson-Thomas type identities for motivic characteristic classes}
One outcome of Donaldson-Thomas theory for Dynkin quivers is an identity (for every quiver), originally due to Reineke (see more detail in Remark \ref{ex:qdilog} below). Both sides of this identity is a product of quantum dilogarithm series, on one side the factors are parameterized by simple roots, and on the other side the factors are parameterized by positive roots; the identity for $A_2$ is hence often called the pentagon identity. The values are evaluated in a quantum algebra. 

In Theorem \ref{thm:coincidence} we prove that our presentation of CSM and MC classes can be translated to an identity whose structure is similar to the quantum dilogarithm identities. On both sides of the identity we have (generalized) exponentials of the $c^o_\beta$, $C^o_\beta$ classes, on one side the factors are parameterized by simple roots, on the other side they are parameterized by positive roots. Now the values are evaluated in the CoHA, KHA instead of the quantum algebra. The quantum dilogarithm identities boil down to identities among rational functions in one variable. Our identities boil down to identities among rational functions in infinitely many variables. 

The {\bf third achievement} of the paper is Theorem \ref{thm:Z}, a generalization of Theorem \ref{thm:coincidence}. Here---following the pattern of Donaldson-Thomas invariants---we define a product of exponentials of $c^o_\beta$ ($C^o_\beta$) classes depending on {\em stability functions}, and prove that the value is independent of the stability function. Again, the encoded identities are among rational functions in infinitely many variables. 

\subsection{Computation of the $c^o_\beta$, $C^o_\beta$ characteristic classes}
It is clear now that the main ingredients of the theory of CSM and MC classes for quivers, as well as the building blocks of the CoHA and KHA are the $c^o_\beta$, $C^o_\beta$ characteristic classes. Although there are not many of them, just one for every positive root for every Dynkin quiver, their structure is rich. In Section \ref{sec:coCo} we present three approaches to their calculation. First, we describe the brute force {\em sieve} method, and we show an improved version making it not completely computationally hopeless. Then we recall results form \cite{FRcsm, FRW} showing {\em interpolation characterizations} for $c^o_\beta$ and $C^o_\beta$. This interpolation characterization is strongly motivated by Okounkov's axioms for cohomological and K-theoretic stable envelopes---although in K-theory the main Newton-polygon axiom is weaker than that of Okounkov. The {\bf fourth  
achievement} of the paper is Theorem \ref{thm:commutator}, in which we present the $c^o_\beta$, $C^o_\beta$ classes as certain commutators in the CoHA, KHA. We illustrate the efficiency of this theorem by examples, and we phrase Conjecture \ref{con:key} for the explicit value of~$c^o_\beta$.

\subsection{Acknowledgment}
The research presented in this paper grew out of discussions with L.~Mihalcea on motivic characteristic classes. The author is very grateful to him for valuable remarks, questions, and suggestions. The author acknowledges the support of Simons Foundation grant~52388.

\section{Quivers}

\noindent Our general reference for the theory of quivers is \cite{Ki}.

\subsection{Quiver representation spaces}
Let $Q=(Q_0,Q_1)$ be an oriented graph---the quiver---with vertex set $Q_0$ and edge set $Q_1$. An edge $a\in Q_1$ has a tail $t(a)\in Q_0$ and a head $h(a)\in Q_0$. Vectors of the type $\N^{Q_0}$ are called dimension vectors. For a dimension vector $\gamma\in \N^{Q_0}$ we will be concerned with the vector space
\[
\Rep^Q_\gamma=\Rep_\gamma=\bigoplus_{a\in Q_1} \Hom(\C^{\gamma(t(a))},\C^{\gamma(h(a))})
\]
acted upon by the group
\[
\GL_\gamma=\prod_{i \in Q_0} \GL_{\gamma(i)}(\C)
\]
by the rule
\[
(A_i)_{i\in Q_0} \cdot (\phi_a)_{a\in Q_1} =(  A_{h(a)}  \circ \phi_a \circ A_{t(a)}^{-1} )_{a\in Q_1}.
\]

\subsection{Path algebra, modules} \label{sec:pathalg}

Consider the vector space of formal complex linear combinations of oriented paths is $Q$, including the length zero path at each vertex $i\in Q_0$. Endowing this vector space with the multiplication induced by concatenation of oriented paths is called the path algebra $\CQ$ of the quiver.

An important fact in the theory of quivers is that the isomorphism types of finite dimensional right modules over $\CQ$ (aka. quiver representations) are in bijection with the orbits of the representations $\{\Rep^Q_\gamma\}_{\gamma \in \N^{Q_0}}$. 

\subsection{Dynkin quivers}
If the underlying un-oriented graph of $Q$ is one of the simply-laced Dynkin graphs $A_n, D_n, E_6, E_7, E_8$, then the quiver is called a Dynkin quiver. For Dynkin quivers Gabriel's theorem describes all $\CQ$-modules,
as follows. Let $\ep_i$, $i\in Q_0$ be the simple roots of the same named root system, and let $R^+=\{\beta_j\}_{j=1,\ldots,N}$ be the set of  positive roots.
\begin{itemize}
\item The (isomorphism types of) indecomposable $\CQ$-modules are in bijection with $R^+$.
\item The dimension vector of the indecomposable $\CQ$-module $M_{\beta_j}$ corresponding to $\beta_j=\sum_{i\in Q_0} \beta_{j,i}\ep_i\in R^+$ is $\beta_{j,i}\in \N^{Q_0}$. Below we identify this dimension vector with $\beta$ itself, that is, a positive root $\beta$ will be thought of a dimension vector.
\item Every $\CQ$-module is the direct sum of indecomposables in a unique way.
\end{itemize}

\subsection{Orbits of Dynkin quivers by Kostant partitions or diagrams}

Section \ref{sec:pathalg} reduced the description of $\Rep_\gamma$ orbits to the description of $\CQ$-modules, and Gabriel's theorem described the $\CQ$-modules for Dynkin quivers. As a corollary we obtain the orbit description of Dynkin quivers, as follows. Let $Q$ be a Dynkin quiver, and let $\gamma$ be a dimension vector. The orbits of $\Rep_\gamma$ are in bijection with so-called {\em Kostant partitions} of $\gamma$, that is, $(m_\beta)_{\beta\in R^+}$ vectors satisfying
$\sum_{\beta\in R^+} m_\beta \beta=\gamma$.

\begin{example} \rm
For $A_2=(1 \to 2)$ the positive roots are $\beta_1=\ep_1$, $\beta_2=\ep_1+\ep_2$, $\beta_3=\ep_2$. For dimension vector $\gamma=(3,4)$ hence we have the Kostant partitions $(0,3,1)$, $(1,2,2)$, $(2,1,3)$, $(3,0,4)$. The corresponding orbits of $\Rep_{(3,4)}$ can be visualized by the diagrams
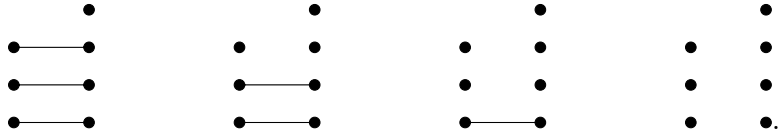
\[
\begin{tikzpicture}
\filldraw (0,0) circle (2pt) -- (1,0) circle (2pt);
\filldraw (0,.5) circle (2pt) -- (1,.5) circle (2pt);
\filldraw (0,1) circle (2pt) -- (1,1) circle (2pt);
\filldraw                            (1,1.5) circle (2pt);

\filldraw (3,0) circle (2pt) -- (4,0) circle (2pt);
\filldraw (3,.5) circle (2pt) -- (4,.5) circle (2pt);
\filldraw (3,1) circle (2pt)  (4,1) circle (2pt);
\filldraw                            (4,1.5) circle (2pt);

\filldraw (6,0) circle (2pt) -- (7,0) circle (2pt);
\filldraw (6,.5) circle (2pt) (7,.5) circle (2pt);
\filldraw (6,1) circle (2pt)  (7,1) circle (2pt);
\filldraw                            (7,1.5) circle (2pt);

\filldraw (9,0) circle (2pt)   (10,0) circle (2pt);
\filldraw (9,.5) circle (2pt)  (10,.5) circle (2pt);
\filldraw (9,1) circle (2pt)  (10,1) circle (2pt);
\filldraw                            (10,1.5) circle (2pt);
\end{tikzpicture}.
\]
Of course, these orbits are the linear maps $\C^3\to \C^4$ of rank $3,2,1,0$ respectively.
\end{example}

\begin{example} \label{ex:A3} \rm
For $A_3=(1\to 2 \to 3)$ the positive roots are  $\beta_1=\ep_1$, $\beta_2=\ep_1+\ep_2$, $\beta_3=\ep_1+\ep_2+\ep_3$, $\beta_4=\ep_2$, $\beta_5=\ep_2+\ep_3$, $\beta_6=\ep_3$. For dimension vector $\gamma=(1,2,1)$ the Kostant partitions are $(0,0,1,1,0,0), (0,1,0,0,1,0), (0,1,0,1,0,1), (1,0,0,1,1,0)$, $(1,0,0,1,0,1)$ visualized by the diagrams
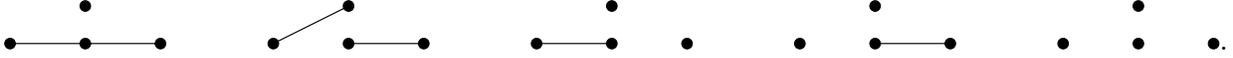
\[
\begin{tikzpicture}
\filldraw (0,0) circle (2pt) -- (1,0) circle (2pt) --(2,0) circle (2pt);
\filldraw (1,.5) circle (2pt);

\filldraw (3.5,0) circle (2pt) -- (4.5,0.5) circle (2pt) ;
\filldraw (4.5,0) circle (2pt) -- (5.5,0) circle (2pt);

\filldraw (7,0) circle (2pt) -- (8,0) circle (2pt)  (9,0) circle (2pt);
\filldraw (8,.5) circle (2pt);

\filldraw (10.5,0) circle (2pt)  (11.5,0) circle (2pt) --(12.5,0) circle (2pt);
\filldraw (11.5,.5) circle (2pt);

\filldraw (14,0) circle (2pt) (15,0) circle (2pt)  (16,0) circle (2pt);
\filldraw (15,.5) circle (2pt);
\end{tikzpicture}.
\]
For example, the first (second) orbit consist of pairs of rank 1 maps $(\psi_1:\C\to \C^2,\psi_2:\C^2\to \C)$ for which $\psi_2\psi_1=0$ (respectively $\psi_2\psi_1=$isomorphism).
\end{example}

\subsection{The Reineke order of positive roots of a Dynkin quiver}

Let $Q$ be a Dynkin quiver. Recall that for a positive root $\beta$ we denote by $M_\beta$ the corresponding indecomposable module (whose dimension vector is $\beta$). For positive roots $\beta_1,\beta_2$ define the Euler form
\[
\chi(\beta_1,\beta_2)= 
\dim \Hom_{\CQ}(M_{\beta_1},M_{\beta_2})
-\dim \Ext^1_{\CQ}(M_{\beta_1},M_{\beta_2}).
\]
It is a well known fact that the Euler form can be calculated from the dimension vectors $\beta_1,\beta_2$ themselves as
\[
\chi(\beta_1,\beta_2)=
\sum_{i\in Q_0} \beta_1(i)\beta_2(i) 
-\sum_{a\in Q_1} \beta_1(t(a)) \beta_2(h(a)).
\]

\begin{definition}
For a Dynkin quiver $Q$ we call $\beta_1 < \beta_2 < \ldots < \beta_N$ a Reineke order of the positive roots if
\begin{equation}\label{eqn:ordercond}
i>j \qquad \Rightarrow \qquad \Hom_{\CQ}(M_{\beta_j},M_{\beta_i})=0 \ \text{and}\ \Ext_{\CQ}^1(M_{\beta_i},M_{\beta_j})=0.
\end{equation}
\end{definition}

This definition is a variant of \cite[Definition 2]{Re}---with the two differences that we only consider total orders not the general partial orders, and we reversed the conventions in \cite{Re}. It is also shown there that condition (\ref{eqn:ordercond}) is equivalent to
\[
i>j \qquad \Rightarrow \qquad \chi(M_{\beta_i},M_{\beta_j})\geq 0 \geq \chi(M_{\beta_j},M_{\beta_i}),
\]
and that  Reineke orders exist. In fact, a conceptual way of finding Reineke orders is using the ``knitting algorithm'' to construct the Auslander-Reiten graph of $Q$ (see an example in \cite[Section 6]{tegez}) and using the combinatorics of that graph---we will not cover this argument here.

\begin{example} \rm
For the $A_2=(\circ \to \circ)$ quiver the only Reineke order is $(1,0)<(1,1)<(0,1)$. For the equioriented $A_3=(\circ \to \circ \to \circ)$ quiver a Reineke order is
\[
(1,0,0)<(1,1,0)<(0,1,0)<(1,1,1)<(0,1,1)<(0,0,1),
\]
and another one is obtained by replacing $(0,1,0)$ and $(1,1,1)$.
\end{example}

\section{Geometric constructions based on a sequence of dimension vectors} \label{sec:GEO}

Let $Q$ be a quiver and let $\ggamma=(\gamma_1,\ldots,\gamma_r)$ be a list (i.e. an ordered sequence) of dimension vectors, and set $\gamma=\sum_{u=1}^r \gamma_u$. Note the difference between $\ggamma$ (a list of dimension vectors) and $\gamma$ (a dimension vector). In this section we will associate some geometric constructions to $\ggamma$.

\subsection{The bundles} \label{sec:bundles}
For $i\in Q_0$ define $\Fl_{\ggamma,i}$ to be the partial flag variety parameterizing chains of subspaces
\begin{equation}\label{eqn:chain}
0=V_{i,0} \subset V_{i,1} \subset V_{i,2}\subset \ldots \subset V_{i,r}=\C^{\gamma(i)}
\end{equation}
with $\dim V_{i,u}=\gamma_1(i)+\ldots+\gamma_u(i)$. Let $\SSS_{i,u}$ be the bundle over $\Fl_{\ggamma,i}$ whose fiber over the point~(\ref{eqn:chain}) is $V_{i,u}/V_{i,u-1}$. Define
\[
\Fl_{\ggamma}=\prod_{i\in Q_0} \Fl_{\ggamma,i}.
\]
\begin{definition}
Let
\begin{align}
\G^{(1)}_{\ggamma}&=\bigoplus_{a\in Q_1} \bigoplus_{v<w} \Hom(\SSS_{t(a),v},\SSS_{h(a),w}), \nonumber  \\
\G^{(2)}_{\ggamma}&=\bigoplus_{a\in Q_1} \bigoplus_{v>w} \Hom(\SSS_{t(a),v},\SSS_{h(a),w}), \label{eqn:Gdef} \\
\G^{(3)}_{\ggamma}&=\bigoplus_{i\in Q_0} \bigoplus_{v<w} \Hom(\SSS_{i,v},\SSS_{i,w}), \nonumber\\
\G^{(4)}_{\ggamma}&=\bigoplus_{a\in Q_1} \bigoplus_{v=1}^r \Hom(\SSS_{t(a),v},\SSS_{h(a),v}), \nonumber
\end{align}
be the bundles pulled back from $\Fl_{\ggamma}$ to $\Fl_{\ggamma}\times \Rep_\gamma$ via the projection map.
\end{definition}

\subsection{The incidence variety}

If $(V_{i,u})_{i\in Q_0,u=1,\ldots,r}\in\Fl_{\ggamma}$ and $(\phi_a)_{a\in Q_1}\in \Rep_\gamma$ satisfy $\phi_a(V_{t(a),u})\subset V_{h(a),u}$ for all $a$ and $u$, then we will call them {\em consistent}.
For such a consistent pair we have the induced linear maps
\[
\tilde{\phi}_{a,u}\!:\!V_{t(a),u}/V_{t(a),u-1} \to V_{h(a),u}/V_{h(a),u-1}
\]
for all $a$ and $u$. The collection $(\tilde{\phi}_{a,u})_{a\in Q_1}$ for a fixed $u$ is an element of $\Rep_{\gamma_u}$.

\begin{definition} \label{def:incvar}
Let $Q$ be a Dynkin quiver. Define the {\em incidence variety}
\begin{align*}
\Sigma_{\ggamma}=
\{
( (V_{i,u}),(\phi_a))\in \Fl_{\ggamma}\times \Rep_\gamma
\ |\
&
(V_{i,u}) \text{ and } (\phi_a) \text{ are consistent}, \\
&
(\tilde{\phi}_{a,u})_{a\in Q_1} \text{ is in the open orbit of } \Rep_{\gamma_u} \text{ for all } u
\}.
\end{align*}
\end{definition}
We will refer to the two conditions in the definition as the `consistency' condition and the `openness' condition.
Here are some examples of descriptions of open orbits.
\begin{itemize}
\item For $Q=A_2$ the open orbit consists of maximal rank maps.
\item Let $Q=D_4$ with edges oriented towards the center (the degree 3 vertex). Let $\gamma=(1,1,2,1)$ (the third coordinate being the center). The open orbit consists of three {\em injective} linear maps $\C^1\to \C^2$ such that their images are {\em three different} lines in $\C^2$.
\item \cite{BR} Let $Q$ be a type $A$ quiver with some orientation, and $\gamma$ a dimension vector. The diagram of the open orbit is obtained as follows. For each vertex $i$ consider $\gamma(i)$ dots below each other in such a way that neighboring dot towers are aligned at the top if the edge is oriented left, and are aligned at the bottom if the edge is oriented right. Then drawing all possible horizontal lines between dots will result in the diagram of the open orbit. For example, for $Q=(\circ \leftarrow \circ \to \circ \leftarrow \circ \to \circ \leftarrow\circ)$ and dimension vector $(2,3,2,4,6,3)$ the open orbit is
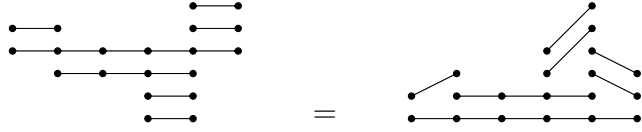
\[
\begin{tikzpicture}[scale=0.6]
\filldraw (3,0) circle (2pt) -- (4,0) circle (2pt);
\filldraw (3,.5) circle (2pt) -- (4,.5) circle (2pt);
\filldraw (1,1) circle (2pt) -- (2,1) circle (2pt) -- (3,1) circle (2pt) -- (4,1) circle (2pt);
\filldraw (0,1.5) circle (2pt) -- (1,1.5) circle (2pt) -- (2,1.5) circle (2pt) -- (3,1.5) circle (2pt) -- (4,1.5) circle (2pt) -- (5,1.5) circle (2pt);
\filldraw (0,2) circle (2pt) -- (1,2) circle (2pt);
\filldraw (4,2) circle (2pt) -- (5,2) circle (2pt);
\filldraw (4,2.5) circle (2pt) -- (5,2.5) circle (2pt);
\end{tikzpicture}
\qquad
=
\qquad
\begin{tikzpicture}[scale=0.6]
\filldraw (0,0) circle (2pt) -- (1,0) circle (2pt) -- (2,0) circle (2pt) -- (3,0) circle (2pt) -- (4,0) circle (2pt) -- (5,0) circle (2pt);
\filldraw (0,.5) circle (2pt) -- (1,1) circle (2pt);
\filldraw (1,.5) circle (2pt) -- (2,.5) circle (2pt) -- (3,.5) circle (2pt) -- (4,.5) circle (2pt);
\filldraw (3,1) circle (2pt) -- (4,2) circle (2pt);
\filldraw (3,1.5) circle (2pt) -- (4,2.5) circle (2pt);
\filldraw (4,1) circle (2pt) -- (5,.5) circle (2pt);
\filldraw (4,1.5) circle (2pt) -- (5,1) circle (2pt);
\end{tikzpicture}.
\]
\end{itemize}

\subsection{Incidence variety elements as filtrations of $\CQ$-modules} \label{sec:filt}

Consider an element of $A=((V_{i,u}),(\phi_a))\in \Fl_{\ggamma}\times \Rep_\gamma$ satisfying the consistency condition above. Its projection to $\Rep_\gamma$ corresponds to a $\CQ$-module, let us call it $M$. The element $A$ then can be thought of as a filtration 
\begin{equation}\label{eqn:filt}
0=M^{(0)}\subset M^{(1)} \subset M^{(2)} \subset \ldots \subset M^{(r)}=M,
\end{equation}
of $\CQ$-modules. Namely, $M^{(u)}$ is the $\CQ$-module obtained from the quiver representation 
\[
\left( (V_{u,i})_{i\in Q_0}, (\phi_a|_{V_{u,t(a)}})_{a\in Q_1}\right).
\]

The element $A$ satisfies the openness condition (that is $A\in\Sigma_{\ggamma}$) if the subquotient $\CQ$-modules $M^{(u)}/M^{(u-1)}$ correspond to the open orbit of $\Rep_{\gamma_u}$.

\subsection{Fibration between two incidence varieties} Let $\Fl(m)$ denote the full flag variety on~$\C^m$.

\begin{lemma} \label{lem:fibr}
For $\beta_1,\ldots,\beta_r\in R^+$ and  $m_1,\ldots,m_r\in \N$, let $\sum_u m_u\beta_u=\gamma$ and 
\begin{align}
\ggamma=&
( m_1\beta_1,m_2\beta_2,\ldots,m_r\beta_r),
\notag \\
\ggamma'=&
(\underbrace{\beta_1,\ldots, \beta_1}_{m_1},
\underbrace{\beta_2,\ldots, \beta_2}_{m_2},
\ldots,
\underbrace{\beta_r,\ldots, \beta_r}_{m_r}).
\label{eqn:lists}
\end{align}
Let $\Phi:\Fl_{\ggamma'} \times \Rep_\gamma \to \Fl_{\ggamma}\times \Rep_\gamma$ be induced by the obvious forgetful map $\Fl_{\ggamma'} \to \Fl_{\ggamma}$. Then
\begin{enumerate}[(i)]
\item \label{i} $\Phi$ maps $\Sigma_{\ggamma'}$ into $\Sigma_{\ggamma}$;
\item \label{ii} The obtained map $\Phi: \Sigma_{\ggamma'} \to \Sigma_{\ggamma}$ is a fibration with fiber 
\[
\Fl(m_1)\times \Fl(m_2) \times \ldots \times \Fl(m_r).
\]
\end{enumerate}
\end{lemma}

\begin{proof}
In the proof we will use the following results from the theory of quivers. 
\begin{enumerate}[(a)]
\item \label{a} \cite[Theorem 2.4]{Ki} Let $M$ be the $\CQ$-module corresponding to the orbit $\OO\subset \Rep_\gamma$. Then $\codim(\OO \subset \Rep_Q)=\dim(\Ext(M,M))$ (Artin-Voigt lemma).
\item \label{b} \cite[Theorem 3.25]{Ki} For a positive root $\beta$ we have $\Hom(M_\beta,M_\beta)=\C$, $\Ext(M_\beta,M_\beta)=0$.
\end{enumerate}
Consider the length $\sum m_u$ filtration of $\CQ$-modules corresponding to the element $A=$ $((V_{i,u}),$ $(\phi_a)) \in \Sigma_{\ggamma'}$ (cf.~Section~\ref{sec:filt}), and call it ``first filtration''. Its subquotients are isomorphic to appropriate $M_{\beta_u}$'s. The $\Phi$-image of $A$ obviously satisfies the consistency condition, hence it corresponds to a filtration of  $\CQ$-modules---call is the ``second filtration''. The second filtration is obtained from the first one by considering only the $m_1$'th, $m_1+m_2$'th, etc elements of the first filtration. 

Since $M_{\beta_u}$ modules have only trivial extensions, see  (\ref{b}), the subquotients of the second filtration are isomorphic to $\oplus^{m_u} M_{\beta_u} =m_uM_{\beta_u}$'s. From the bilinearity of $\Ext$ and (\ref{b}) we get that $\Ext(mM_\beta,mM_\beta)=0$ and hence---using (\ref{a})---we get that $m_uM_{\beta_u}$ corresponds to the open orbit of $\Rep_{m_u\beta_u}$. Therefore $\Phi(A)\in \Sigma_\gamma$, and (\ref{i}) is proved. To proceed we need to prove a claim.

\smallskip

\noindent{Claim:} {\em Consider a length $m$ fibration $0=M^{(0)}\subset M^{(1)}\subset \ldots\subset  M^{(m)}=M_{m\beta}$ for a positive root $\beta$ such that each  subquotient is isomorphic to $M_\beta$. Then this fibration is necessarily of the form
\[
0=0\otimes M_\beta \subset V^1\otimes M_\beta \subset V^2\otimes M_\beta \subset \ldots
\subset V^{m-1}\otimes M_\beta\subset \C^m\otimes M_\beta=mM_\beta,
\]
where $V^1\subset V^2\subset \ldots\subset V^{m-1}\subset\C^m$ is a full flag of $\C^m$. Therefore, the set of such fibrations is $\Fl(m)$.}

\smallskip

\noindent{Proof of the Claim:} Indeed, $M^{(m-1)}$ is the kernel of a homomorphism $mM_\beta\to mM_{\beta}/M^{(m-1)}\cong M_\beta$. Since $\Hom(mM_\beta,M_\beta)=\C^m$ (see (\ref{b})) there are exactly those homomorphisms $mM_\beta\to M_\beta$ whose kernels are of the form $V^{m-1}\otimes M_\beta$. Iterating this argument proves the claim. 

\smallskip

To prove (\ref{ii}) we just need to apply the claim to find the set of extensions of the length $r$ filtration of $\CQ$-modules corresponding to an element of $\Sigma_{\ggamma'}$, to obtain a length $\sum m_j$ filtration corresponding to an element of $\Sigma_{\ggamma}$.
\end{proof}

\subsection{Reineke's resolution of Dynkin quiver orbits} \label{sec:incvars}

In this section we will study three incidence varieties \`a la Definition \ref{def:incvar}, and their relations to quiver orbits.

Fix a Dynkin quiver $Q$ and a Reineke order $\beta_1<\ldots<\beta_N$ of its positive roots. Let $m=(m_1, \ldots, m_N)$ be a vector of non-negative integers. Then $m$ is a Kostant partition of the dimension vector $\gamma=\sum_{j=1}^N m_j \beta_j$. Let $\OO_m$ be the corresponding orbit in $\Rep_\gamma$.

Let us fix a total order $i_1<i_2<\ldots<i_n$ on $Q_0$ satisfying $h(a)<t(a)$ for all $a\in Q_1$. For a positive root $\beta_j$ and a vertex $i\in Q_0$ let $d_{j,i}$ be the dimension vector which is $\beta_j(i)$ at vertex $i$ and is 0 at all other vertex.

Consider the following three lists of dimension vectors:
\begin{align}
\ddelta_m=&
( m_1\beta_1,m_2\beta_2,\ldots,m_N\beta_N),
\notag \\
\ddelta'_m=&
(\underbrace{\beta_1,\ldots, \beta_1}_{m_1},
\underbrace{\beta_2,\ldots, \beta_2}_{m_2},
\ldots,
\underbrace{\beta_N,\ldots, \beta_N}_{m_N}),
\label{eqn:lists}\\
\ddelta''_m=&
(m_1d_{1,1},m_1d_{1,2},\ldots,m_1d_{1,n},\ \ \
m_2d_{2,1},m_2d_{2,2},\ldots,m_2d_{2,n},\ \ \ \ldots,
\notag \\
&
\ \hskip 6 true cm \ \ \  m_Nd_{N,1},m_Nd_{N,2},\ldots,m_Nd_{N,n}).
\notag
\end{align}
[The third list also depends of the total order of vertices chosen, we will not indicate this dependence in notation.]
Observe that the sum of the dimension vectors in each list is $\gamma=\sum_j m_j\beta_j$, the dimension vector $\gamma$ for which $\OO_m\subset \Rep_\gamma$.

\begin{example} \rm
For $Q=A_2=(1\to 2)$ the (only) Reineke order of positive roots is $(1,0)<(1,1)<(0,1)$. 
For $m=(2,2,2)$ we have $\OO_m=\{\rk 2 \text{ maps }\C^4\to \C^4 \}\subset \Rep_{(4,4)}$, and the three lists of dimension vectors are
\begin{align*}
\ddelta_m=&
\left( (2,0), (2,2), (0,2) \right),
\\
\ddelta'_m=&
\left(
(1,0),(1,0),(1,1),(1,1),(0,1),(0,1)
\right),
\\
\ddelta''_m=&
\left(
(0,0),(2,0),(0,2),(2,0),(0,2),(0,0)
\right).
\end{align*}
\end{example}


In the definition of $\Sigma_{\ddelta''_m}$ there are two conditions: the `consistency' condition, and the `openness' conditions. Observe that the latter one is an empty condition because the dimension vector of all occurring $(\tilde{\phi}_{a,u})_{a\in Q_1}$ is simple, that is, 0 in all but one component. Therefore, for $\ddelta''$, the incidence variety is the same as is defined in \cite{Re}---where only the consistency condition is required. For this incidence variety Reineke's result is the following.

\begin{proposition}\cite[Theorem 2.2]{Re} \label{prop:reineke}
The restriction $\pi$ of the projection $\pi_2:\Fl_{\ddelta_m''} \times \Rep_\gamma\to \Rep_\gamma$ to $\Sigma_{\ddelta''_m}$ has image $\overline{\OO}_m$. The map $\pi$ restricts to an isomorphism $\pi^{-1}(\OO_m)\to \OO_m$. \qed
\end{proposition}

\begin{proposition} \label{prop:resolution2}
The projection $\pi_2:\Fl_{\ddelta_m} \times \Rep_\gamma \to \Rep_\gamma$, restricted to $\Sigma_{\ddelta_m}$ is an isomorphism $\Sigma_{\ddelta_m}\to \OO_m$.
\end{proposition}

\begin{proof}
This statement is contained in the proof \cite[Theorem 2.2]{Re}. 
\end{proof}

An alternative proof of Propositions \ref{prop:reineke}, \ref{prop:resolution2} can be obtained using \cite{King}---this argument will be carried out in a more general context in our proof of Theorem \ref{thm:Z} below.

\begin{proposition} \label{prop:resolution3}
The projection $\pi_2:\Fl_{\ddelta'_m} \times \Rep_\gamma \to \Rep_\gamma$, restricted to $\Sigma_{\ddelta'_m}$ is a fibration with fiber
\[
\Fl(m_1)\times \Fl(m_2) \times \ldots \times \Fl(m_N).
\]
over $\OO_m$, where $\Fl(m)$ denotes the full flag variety on $\C^m$.
\end{proposition}

\begin{proof}
This follows from Proposition \ref{prop:resolution2}, using Lemma \ref{lem:fibr}.
\end{proof}

\section{Motivic charactersitic classes in cohomology and K theory}  

\subsection{Equivariant Chern-Schwartz-MacPherson class}\label{sec:introCSM}
Let the algebraic group $G$ act on the smooth algebraic variety $V$, and let $f:V\to \C$ be an invariant constructible function. Following works of Deligne, Grothendieck, MacPherson \cite{M}, Ohmoto \cite{O1, O2} defined the equivariant Chern-Schwartz-MacPherson (CSM) class $\csm(f)\in H_G^*(V)$.  If $X$ is an invariant constructible subset of $V$ and $\One_X$ is its indicator function, then one defines $\csm(X \subset V)=\csm(\One_X)$. If $V$ is clear from the context we write $\csm(X)$ for $\csm(X\subset V)$.

For reviews on the non-equivariant and the equivariant CSM classes and their role in algebraic geometry see \cite{AM1, AM2, O1, O2, FRcsm, AMSS}.
Here are some properties of CSM classes.
\begin{itemize}
\item (additivity) For the linear combination $af+bg$ of invariant constructible functions $f,g$ we have $\csm(af+bg)=a\csm(f)+b\csm(g)$.
\item (normalization) For a smooth, proper subvariety $i:X\subset V$ we have $\csm(X)=i_*(c(TX))$.
\item (push-forward property) If $\eta:W\to V$ is a proper equivariant morphism between smooth varieties then 
$ \eta_*(c(TW))= \sum_a a\csm(V_a)$, where
$V_a=\{x\in V:\chi(\eta^{-1}(x))=a\}$. 
\item (pull-back property) For $W$, $V$ smooth let $X\subset V$ be an invariant subvariety with an invariant Whitney stratification. Assume $\eta:W\to V$ is (equivariant and) transversal to the strata. Then $\csm(\eta^{-1}(X))/c(TW)=\eta^*(\csm(X)/c(TV))$.
\item (product property) We have $\csm(X\times Y \subset V\times W)=\csm(X\subset V)\csm(Y\subset W)$.
\item (integral property) For a smooth compact $V$ we have $\int_V \csm(V\subset V)=\chi(V)$.
\end{itemize} 

We will use these properties of CSM classes in the paper without reference. To get familiar with our notation consider the two-torus $\T=(\C^*)^2$ acting on $\C^2$ diagonally. Denote the $x$-axis ($y$-axis)---as a subset or as a bundle over the origin---by $X$ ($Y$). We have
\[
\csm(X\subset \C^2)=c(TX)e(Y)=(1+\alpha_1)\alpha_2 \  \in H^*_{  \T }(\C^2)=H^*(B\T )=\Z[\alpha_1,\alpha_2].
\]

\begin{example}\label{ex:coordinateplanes} \rm
Consider a torus $\T$ acting on $\C^n$, with (additive) weights $w_1,\ldots,w_n$ on the coordinate axes, i.e. $w_i\in H^2_{\T}(\C^n)=H^2(B\T)$. For $I\subset \{1,\ldots,n\}$ denote
\[
X_I=\{x\in \C^n : x_i=0 \text{ for } i\in I\},
\qquad
X^o_I=\{x\in \C^n : x_i=0 \text{ for } i\in I, x_j\not=0 \text{ for } j\not\in I\}.
\]
In $\T$-equivariant cohomology we have
\[
\csm(X_I)=\prod_{i\in I} w_i\prod_{j\not\in I} (1+w_j), \qquad
\csm(X^o_I)=\prod_{i\in I} w_i.
\]
In particular, $\csm(X^o_{\emptyset})=1$. 
\end{example}

\subsection{Equivariant motivic Chern class}
\label{sec:introCSM}
Let the algebraic group $G$ act on the smooth algebraic variety $V$, and let $f:X\to V$ be a morphism, and let $y$ be a variable. Following works of Brasselet, Sch{\"u}rmann, Yokura \cite{BSY},  the equivariant motivic Chern (MC) class $\mC(X\to V)\in K_G(V)[y]$ is defined in \cite{FRW, AMSS2} (see also the related \cite{W}). If $X\subset V$ and $V$ is clear from the context we write $\mC(X)$ for $\mC(X\subset V)$.

By K-theoretic total Chern class of a bundle $\xi$ we mean $c(\xi)=\lambda_y(\xi^*)=\oplus_{p=0}^{rk \xi} [\Lambda^p\xi^*]y^p$, and by K-theoretic Euler class we mean $e(\xi)=c(\xi)|_{y=-1}=\lambda_{-1}(\xi^*)$.  For example for a line bundle $L$ in K-theory we have $c(\xi)=1+y/\alpha$, $e(\xi)=1-1/\alpha$ where $\alpha$ is the class represented by $L$. 

Here are some properties of MC classes.
\begin{itemize}
\item (additivity) If $X=Y \sqcup U$ then $\mC(X\to V)=\mC(Y\to V)+\mC(U\to V)$.
\item (normalization) For a smooth, proper subvariety $i:X\subset V$ we have $\mC(X)=i_*(c(TX))$. 
\item (push-forward property) For a proper map $f:V\to V'$ we have $\mC(X\to V')=f_*\mC(X\to V)$.
\item (product property) We have $\mC(X\times Y \subset V\times W)=\mC(X\subset V)\mC(Y\subset W)$.
\item (rigidity) For a smooth compact $V$ we have $\int_V \mC(V\subset V)=\chi_y(V)$ (the $\chi_y$-genus of $M$; as usual, $\int_V$ stands for the push-forward map to a point).
\end{itemize} 

A pull-back property (not used in this paper) for MC classes is widely expected too, as a careful application of the Verdier-Riemann-Roch formula, cf. \cite[Remark 5.2]{FRW}, \cite{Sch}. The precise statement, based on the equivariant version of  \cite[Cor. 2.1, part (4)]{BSY}, is announced in \cite{AMSS2}.

We will use these properties of MC classes in the paper without reference. To get familiar with our notation consider the two-torus $\T=(\C^*)^2$ acting on $\C^2$ diagonally. Denote the $x$-axis ($y$-axis)---as a subset or as a bundle over the origin---by $X$ ($Y$). We have
\[
\mC(X\subset \C^2)=c(TX)e(Y)=(1+y/\alpha_1)(1-1/\alpha_2) \qquad \in K_{  \T }(\C^2)=R(\T)=\Z[\alpha_1^{\pm 1},\alpha_2^{\pm 1}].
\]

\begin{example}\label{ex:coordinateplanes_mC} \rm
Consider a torus $\T=(\C^*)^r$ acting on $\C^n$, with multiplicative weights $w_1,\ldots,w_n$ on the coordinate axes, i.e. 
$w_i\in K_{\T}(\C^n)=R(\T)=\Z[\alpha_1^{\pm1},\ldots,\alpha_r^{\pm1}]$. For $I\subset \{1,\ldots,n\}$ denote
\[
X_I=\{x\in \C^n : x_i=0 \text{ for } i\in I\},
\qquad
X^o_I=\{x\in \C^n : x_i=0 \text{ for } i\in I, x_j\not=0 \text{ for } j\not\in I\}.
\]
In $\T$-equivariant K-theory we have
\[
\mC(X_I)=\prod_{i\in I} \left(1-\frac{1}{w_i}\right) \prod_{j\not\in I}\left( 1+\frac{y}{w_j}\right), \qquad
\mC(X^o_I)=\prod_{i\in I}\left(1-\frac{1}{w_i}\right)\prod_{j\not\in I} \frac{1+y}{w_j}.
\]
In particular, $\mC(X^o_{\emptyset})=(1+y)^n/\prod_{i=1}^n w_i$. 
\end{example}

\section{The CSM and MC classes of incidence varieties and their projections} 
Let $Q$ be a Dynkin quiver and $\gamma$ a dimension vector. We will be concerned with the algebras 
\begin{align} \label{def:HHKK}
\HH^Q_\gamma= &H^*_{\GL_\gamma}(\Rep_\gamma;\C)=
\C[\alpha_{i,u}]_{i\in Q_0, u=1,\ldots,\gamma(i)}^{S_{\gamma}}
\\
\KK^Q_\gamma=&K_{\GL_\gamma}(\Rep_\gamma)\otimes \C[y] 
=
\C[\alpha_{i,u}^{\pm 1},y]_{i\in Q_0, u=1,\ldots,\gamma(i)}^{S_{\gamma}}
\end{align}
where 
\begin{itemize}
\item in the first line $\alpha_{i,1},\ldots,\alpha_{i,\gamma(i)}$ are the Chern roots of the tautological vector bundle over $\BGL_{\gamma(i)}(\C)$;
\item in the second line $\alpha_{i,1},\ldots,\alpha_{i,\gamma(i)}$ are the K-theoretic Chern roots of the group $\GL_{\gamma(i)}(\C)$, that is, $\sum_{u=1}^{\gamma(i)} \alpha_u$ is the standard representation of the group $\GL_{\gamma(i)}(\C)$;
\item $S_{\gamma}=\prod_{i\in Q_0} S_{\gamma(i)}$, the $S_{\gamma(i)}$ factor permutes the $\alpha_{i,*}$ variables.
\end{itemize}

When $Q$ is obvious, we will drop the upper index from $\HH_\gamma^Q$, $\KK^Q_\gamma$. In the rest of this section, after introducing the new characteristic classes $c^o, C^o$, we give formulas for the CSM and MC classes of $\Sigma_\ggamma$, and their push-forwards.

\subsection{The $c^o, C^o$ characteristic classes} \label{sec:co}

\begin{definition}
Let the open orbit of $\Rep_\gamma$ be denoted by $\OO^o_\gamma$. Define
\[
c_\gamma^o=\csm(\OO^o_{\gamma}\subset \Rep_\gamma) \in \HH^Q_\gamma,
\]
\[
C_\gamma^o=\mC(\OO^o_{\gamma}\subset \Rep_\gamma) \in \KK^Q_\gamma.
\]
\end{definition}
Note that we may regard $c^o_\gamma$ or $C^o_\gamma$ as characteristic classes: they can be evaluated on a set of bundles $\{ \E_{i} \}_{i\in Q_0}$ if $\dim \E_i=\gamma(i)$ are given over a space.

\begin{example} \rm \label{exA2}
For $Q=A_2=(1 \to 2)$,
we have the cohomology characteristic classes
\begin{align*}
c^o_{(0,n)}=&1, \qquad c^o_{(n,0)}=1, \qquad c^o_{(1,1)}=1, \\
c^o_{(1,2)}=& 1+\alpha_{2,1}+\alpha_{2,2}-2\alpha_{1,1} \\
c^o_{(2,1)}=& 1+2\alpha_{1,1}-\alpha_{2,1}-\alpha_{2,2} \\
c^o_{(2,2)}=& 1+(\alpha_{2,1}+\alpha_{2,2}-\alpha_{1,1}-\alpha_{1,2})-
              (\alpha_{2,1}+\alpha_{2,2})(\alpha_{1,1}+\alpha_{1,2})+
              2(\alpha_{2,1}\alpha_{2,2}+\alpha_{1,1}\alpha_{1,2}),
\end{align*}
and the K-theoretic characteristic classes
\begin{align*}
C^o_{(0,n)}=&1, \qquad C^o_{(n,0)}=1, \qquad C^o_{(1,1)}=(1+y)\frac{\alpha_{1,1}}{\alpha_{2,1}}, \\
C^o_{(1,2)}=& (1+y)\left(\frac{\alpha_{1,1}}{\alpha_{2,1}}+\frac{\alpha_{1,1}}{\alpha_{2,2}}\right)+(y^2-1)\frac{\alpha_{1,1}^2}{\alpha_{2,1}\alpha_{2,2}} \\
C^o_{(2,1)}=& (1+y)\left(\frac{\alpha_{1,1}}{\alpha_{2,1}}+\frac{\alpha_{1,2}}{\alpha_{2,1}}\right)+(y^2-1)\frac{\alpha_{1,1}\alpha_{1,2}}{\alpha_{2,1}^2}  \\
C^o_{(2,2)}=& (1+y)^2\frac{\alpha_{1,1}\alpha_{1,2}}{\alpha_{2,1}\alpha_{2,2}}
\left(
1-y+y\left(\frac{\alpha_{1,1}}{\alpha_{2,1}}+\frac{\alpha_{1,1}}{\alpha_{2,2}}+\frac{\alpha_{1,2}}{\alpha_{2,1}}+\frac{\alpha_{1,2}}{\alpha_{2,2}}\right)
+y(y-1)\frac{\alpha_{1,1}\alpha_{1,2}}{\alpha_{2,1}\alpha_{2,1}}
\right).
\end{align*}
\end{example}

\begin{example}\rm \label{ex:co1}
Let $Q$ be a Dynkin quiver of type $A$, and let $\beta\in \N^Q$ be a dimension vector corresponding to a positive root. That is, $\beta=\beta_{uv}=(0,\ldots,0,1,\ldots,1,0,\ldots,0)=\sum_{i=u}^v \ep_i$, for some $1\leq u\leq v\leq n$. Then, according to Example \ref{ex:coordinateplanes}, we have $c^o_\beta=1$. 
Let the edge $a_i$ be between vertex $i$ and $i+1$ (for $i=1,\ldots,n-1$) with some orientation.  Then, according to Example \ref{ex:coordinateplanes_mC}, we have
\[
C^o_{\beta_{uv}}=(1+y)^{v-u}\prod_{i=u}^{v-1} \frac{\alpha_{t(a_i),1}}{\alpha_{h(a_i),1}}.
\]
\end{example}

\begin{example} \label{D4D5E6} \rm
Consider the quivers 
\[
\begin{tikzpicture}[scale=0.6]
\node (A) at (-.2,0) {1};
\node (B) at (1,0) {2};
\node (C) at (2.5,0) {3};
\node (D) at (4,0) {4};
\draw [->] (A) to [out=40,in=140] (C);
\draw [->] (B) to [out=20,in=160] (C);
\draw [->] (D) -- (C);
\end{tikzpicture}
\qquad\qquad
\begin{tikzpicture}[scale=0.6]
\node (A) at (-.2,0) {1};
\node (B) at (1,0) {2};
\node (C) at (2.5,0) {3};
\node (D) at (4,0) {4};
\node (E) at (5.5,0) {5};
\draw [->] (A) to [out=40,in=140] (C);
\draw [->] (B) -- (C);
\draw [->] (D) -- (C);
\draw [->] (E) -- (D);
\end{tikzpicture}
\qquad\qquad
\begin{tikzpicture}[scale=0.6]
\node (A) at (-.5,0) {1};
\node (B) at (1,0) {2};
\node (C) at (2.5,0) {3};
\node (D) at (4,0) {4};
\node (E) at (5.5,0) {5};
\node (F) at (6.6,0) {6};
\draw [->] (A) to (B);
\draw [->] (B) to (C);
\draw [->] (D) to (C);
\draw [->] (E) to (D);
\draw [->] (F) to [out=140,in=40] (C);
\end{tikzpicture}
\]
with underlying Dynkin graphs $D_4$, $D_5$, $E_6$---for simplicty let us call them $Q=D_4, D_5, E_6$.
For $Q=D_4$ and $\beta=(1,1,2,1)$ we have 
\[
c^o_\beta=(1+\alpha_{3,1}-\alpha_{3,2})(1+\alpha_{3,2}-\alpha_{3,1}),
\]
and temporarily using the notation $a=\alpha_{1,1}, b=\alpha_{2,1},c=\alpha_{4,1},d_i=\alpha_{3,i}$ we have
\[
C^o_\beta=
\frac{(1+y)^4abc}{(d_1d_2)^3}
\left(
(1-y)(a+b+c)d_1d_2+(d_1+d_2)
(y(ab+ac+bc)-d_1d_2)-
y(1-y)abc
\right).
\]
For $Q=D_5$ and $\beta=(1,1,2,2,1)$ we have
\[
c^o_\beta=(1+\alpha_{3,1}-\alpha_{3,2})(1+\alpha_{3,2}-\alpha_{3,1})(1+\alpha_{4,1}-\alpha_{4,2})(1+\alpha_{4,2}-\alpha_{4,1}).
\]
For $Q=E_6$ and $\beta=(1,2,3,2,1,2)$ we have
\begin{multline*}
c^o_\beta=(1+\alpha_{2,1}-\alpha_{2,2})(1+\alpha_{2,2}-\alpha_{2,1})(1+\alpha_{4,1}-\alpha_{4,2})(1+\alpha_{4,2}-\alpha_{4,1})(1+\alpha_{3,1}-\alpha_{3,2})\\
(1+\alpha_{3,2}-\alpha_{3,1})(1+\alpha_{3,1}-\alpha_{3,3})(1+\alpha_{3,3}-\alpha_{3,1})(1+\alpha_{3,2}-\alpha_{3,3})(1+\alpha_{3,3}-\alpha_{3,2}).
\end{multline*}
\end{example}

For the proof of these claims, and more discussion on $c^o_\beta$ and $C^o_\beta$ classes see Section \ref{sec:coCo}.

\subsection{The CSM class of $\Sigma_{\ggamma}$}


\begin{proposition} \label{prop:classofSigma}
For a list of dimension vectors $\ggamma=(\gamma_1,\ldots,\gamma_r)$ with $\gamma=\sum_{u=1}^r \gamma_u$, in $\GL_\gamma$-equivariant cohomology and K theory we have
\begin{align}\label{eqn:inprop}
\csm(\Sigma_{\ggamma}\subset \Fl_{\ggamma}\times \Rep_\gamma)=
e(\G_\ggamma^{(1)}) c(\G_\ggamma^{(2)}) c(\G_\ggamma^{(3)})
\prod_{u=1}^r c_{\ggamma_u}^o\left( \{\SSS_{i,u}\}_{i\in Q_0} \right),
\\
\mC(\Sigma_{\ggamma}\subset \Fl_{\ggamma}\times \Rep_\gamma)=
e(\G_\ggamma^{(1)}) c(\G_\ggamma^{(2)}) c(\G_\ggamma^{(3)})
\prod_{u=1}^r C_{\ggamma_u}^o\left( \{\SSS_{i,u}\}_{i\in Q_0} \right).
\end{align}
\end{proposition}


\begin{proof}
The group $\GL_\gamma$ acts on $\Fl_\ggamma\times \Rep_\gamma$, with $\Sigma_\ggamma$ invariant. Consider this action restricted to its maximal torus $\T_\gamma=(\C^*)^{\sum_{i\in Q_0} \gamma(i)}$. Let $x=((V_{i,u}),0) \in \Fl_\ggamma\times \Rep_\gamma$ be a point fixed by $\T$. The tangent space $T_x(\Fl_\ggamma\times \Rep_\gamma)$ splits to the direct sum of $\T$-invariant subspaces
\begin{align}
\V_1\ & =\bigoplus_{a\in Q_1} \bigoplus_{v<w} \Hom(V_{t(a),v},V_{h(a),w}), \nonumber  \\
\V_2\ & = \bigoplus_{a\in Q_1} \bigoplus_{v>w} \Hom(V_{t(a),v},V_{h(a),w}), \label{eqn:Vdef} \\
\V_3\ & = T_x\Fl_\ggamma, \nonumber\\
\V_{4,u}& =\bigoplus_{a\in Q_1} \Hom(V_{t(a),u},V_{h(a),u}), \qquad \text{for } u=1,\ldots,r. \nonumber
\end{align}
Consider the following subsets of $\V_1,\V_2,\V_3,\V_{4,u}$
\[
\Y_1=\{0\} \subset \V_1,\qquad
\Y_2=\V_2,\qquad
\Y_3=\V_3,
\]
\[
\Y_{4,u}=\OO_{\gamma_u}=(\text{open orbit of } \Rep_{\gamma_u}) \subset \V_{4,u}.
\]
The definition of $\Sigma_\ggamma$ implies that in a neighborhood of $x$ the identification
$\Fl_\ggamma \times \Rep_\gamma \to T_x(\Fl_\ggamma \times \Rep_\gamma)$ can be chosen so that $\Sigma_\ggamma$ is mapped to the direct product
\[
\Y_1\times \Y_2 \times \Y_3 \times \mathop{\times}_{u=1}^r \Y_{4,u}.
\]
Therefore, the restriction of $\csm(\Sigma_{\ggamma}\subset \Fl_{\ggamma}\times \Rep_\gamma)$ to $x$ in $\T$-equivariant cohomology is
\begin{equation}\label{eqn:therest}
e(\V_1) c(\V_2) c(\V_3) \prod_{u=1}^r c^o_{\gamma_u}\left( (V_{i,u})_{i\in Q_0} \right).
\end{equation}
Here we used the product property of CSM classes, as well as the facts
$\csm(\V\subset \V)=c(\V)$, $\csm(\{0\}\subset \V)=e(\V)$ for a representation $\V$.
The right hand side of (\ref{eqn:inprop}) has the same restriction (\ref{eqn:therest}) to $x$, hence---by equivariant localization---the cohomological statement of the theorem is proved. The K-theoretic statement is proved the same way, by replacing $c^o$ with $C^o$.
\end{proof}

\subsection{Equivariant localization formulas for $\pi_{2*}(\csm(\Sigma_{\ggamma}))$ and $\pi_{2*}(\mC(\Sigma_{\ggamma}))$}
 \label{sec:3part}

For a list of dimension vectors $\ggamma=(\gamma_1,\ldots,\gamma_r)$ with $\sum_{j=1}^r \gamma_j=\gamma$ consider the variables in $\HH^Q_\gamma$ (or $\KK^Q_\gamma$), namely $\alpha_{i,u}$ for $i\in Q_0, u=1,\ldots,\gamma(i)$. By a $\ggamma$-shuffle of these variables we mean $S=(S_i)_{i\in Q_0}$ where
\begin{itemize}
\item $S_i=(S_{i,1},\ldots,S_{i,r})$;
\item $S_{i,1},\ldots,S_{i,r}$ are disjoint subsets of $\{\alpha_{i,1},\ldots,\alpha_{i,\gamma(i)}\}$;
\item $|S_{i,u}|=\gamma_u(i)$ for all $u=1,\ldots,r$.
\end{itemize}
Define
\begin{align*}
\fac_1(S)&=\prod_{a\in Q_1} \prod_{v<w} \prod_{\omega\in {S}_{h(a),w}} \prod_{\alpha\in     {S}_{t(a),v}}  (\omega-\alpha),
&  
\fac'_1(S)&=\prod_{a\in Q_1} \prod_{v<w} \prod_{\omega\in {S}_{h(a),w}} \prod_{\alpha\in     {S}_{t(a),v}}  (1-\alpha/\omega),\\
\fac_2(S)&=\prod_{a\in Q_1} \prod_{v<w} \prod_{\omega\in     {S}_{h(a),v}} \prod_{\alpha\in {S}_{t(a),w}}  (1+\omega-\alpha), 
& 
\fac'_2(S)&=\prod_{a\in Q_1} \prod_{v<w} \prod_{\omega\in     {S}_{h(a),v}} \prod_{\alpha\in {S}_{t(a),w}}  (1+y\alpha/\omega), 
\\
\fac_3(S)&=\prod_{i\in Q_0} \prod_{v<w} \prod_{\omega\in {S}_{i,w}} \prod_{\alpha\in {S}_{i,v}} \frac{1+\omega-\alpha}{\omega-\alpha}, 
& 
\fac'_3(S)&=\prod_{i\in Q_0} \prod_{v<w} \prod_{\omega\in {S}_{i,w}} \prod_{\alpha\in {S}_{i,v}} \frac{1+y\alpha/\omega}{1-\alpha/\omega},\\
\fac_4(S)&=\prod_{u=1}^r c^o_{\gamma_u}\left( \{S_{i,u}\}_{i\in Q_0}\right),
&
\fac'_4(S)&=\prod_{u=1}^r C^o_{\gamma_u}\left( \{S_{i,u}\}_{i\in Q_0}\right).
\end{align*}

\begin{proposition} \label{prop:pi2}
Let $\ggamma=(\gamma_1,\ldots,\gamma_r)$ be a list of dimension vectors with $\sum_{u=1}^r \gamma_u=\gamma$. In $\HH^Q_\gamma$, $\KK^Q_\gamma$  we have
\begin{equation}\label{eqn:pi2H}
\pi_{2*}(\csm(\Sigma_{\ggamma}\subset \Fl_{\ggamma}\times \Rep_\gamma)) =\sum_{S \text{ is a } \ggamma\text{-shuffle}}
\fac_1(S)\fac_2(S)\fac_3(S)\fac_4(S),
\end{equation}
\begin{equation}\label{eqn:pi2K}
\pi_{2*}(\mC(\Sigma_{\ggamma}\subset \Fl_{\ggamma}\times \Rep_\gamma))=\sum_{S \text{ is a } \ggamma\text{-shuffle}}
\fac'_1(S)\fac'_2(S)\fac'_3(S)\fac'_4(S).
\end{equation}
\end{proposition}

\begin{proof}
The formulas follow from Proposition \ref{prop:classofSigma} as the push-forward map $\pi_{2*}$ is evaluated by cohomological or K-theoretic equivariant localization.
\end{proof}

\subsection{The CSM and MC classes of quiver orbits of Dynkin type} \label{sec:CSMclasses}
As before, $Q$ is a Dynkin quiver, $\beta_1<\ldots<\beta_N$ is the Reineke order of positive roots, each identified with its dimension vector. Let $(m_j)_{j=1,\ldots,N}$ be a Konstant partition of the dimension vector $\gamma=\sum_{j=1}^N m_j\beta_j$. Let $\OO_m$ be the corresponding orbit in $\Rep_\gamma$. 

\begin{theorem}[Motivic classes of quiver orbits, version 1] \label{thm:1}
For the list of dimension vectors $\ddelta_m=(m_1\beta_1,m_2\beta_2,\ldots,m_N\beta_N)$ (cf. (\ref{eqn:lists}))) we have
\begin{align*}
\csm(\OO_m \subset \Rep_\gamma)=& 
  \pi_{2*}\left(
  \csm(\Sigma_{\ddelta_m}\subset \Fl_{\ddelta_m}\times \Rep_\gamma)
  \right) 
\\
 =& 
\sum_{S \text{ is a } \ddelta_m\text{-shuffle}}
\fac_1(S)\fac_2(S)\fac_3(S)\fac_4(S);
\\
\mC(\OO_m \subset \Rep_\gamma)=&
\pi_{2*}\left(
\mC(\Sigma_{\ddelta_m}\subset \Fl_{\ddelta_m}\times \Rep_\gamma)
\right),
\\
=& 
\sum_{S \text{ is a } \ddelta_m\text{-shuffle}}
\fac'_1(S)\fac'_2(S)\fac'_3(S)\fac'_4(S),
\end{align*}
\end{theorem}

\begin{proof} The statements follow from Propositions  \ref{prop:resolution2} and \ref{prop:pi2}.
\end{proof}

For $m\in\N$ let $[m]_y!=1(1-y)(1-y+y^2)\ldots(1-y+y^2-y^3+\ldots+(-y)^{m-1})$. The $y=-q$ substitution recovers the usual $q$-factorial notion. 

\begin{theorem}[Motivic classes of quiver orbits, version 2] \label{thm:2}
For the list of dimension vectors 
$\ddelta'_m=(\underbrace{\beta_1,\ldots, \beta_1}_{m_1}, \underbrace{\beta_2,\ldots, \beta_2}_{m_2}, \ldots, \underbrace{\beta_N,\ldots, \beta_N}_{m_N})$ (cf. (\ref{eqn:lists}))) we have
\begin{align*}
\csm(\OO_m \subset \Rep_\gamma)=&
\frac{1}{\prod_{j=1}^r m_j!}\cdot
\pi_{2*}\left(
\csm(\Sigma_{\ddelta'_m}\subset \Fl_{\ddelta'_m}\times \Rep_\gamma)
\right);
\\
=& 
\frac{1}{\prod_{j=1}^r m_j!}
\sum_{S \text{ is a } \ddelta'_m\text{-shuffle}}
\fac_1(S)\fac_2(S)\fac_3(S)\fac_4(S),
\\
\mC(\OO_m \subset \Rep_\gamma)=&
\frac{1}{\prod_{j=1}^r [m_j]_y!}\cdot
\pi_{2*}\left(
\mC(\Sigma_{\ddelta'_m}\subset \Fl_{\ddelta'_m}\times \Rep_\gamma)
\right),
\\
=& 
\frac{1}{\prod_{j=1}^r [m_j]_y!}
\sum_{S \text{ is a } \ddelta'_m\text{-shuffle}}
\fac'_1(S)\fac'_2(S)\fac'_3(S)\fac'_4(S).
\end{align*}
\end{theorem}

\begin{proof} The statements follows from  Proposition  \ref{prop:resolution3} and \ref{prop:pi2}, using the fact that the Euler characteristic of a full flag variety $\Fl(m)$ is $m!$, and the $\chi_y$-genus of the full flag variety $\Fl(m)$ is $[m]_y!$.
\end{proof}

\noindent In Dynkin type $A$, and in cohomology the only non-explicit factor in Theorem \ref{thm:2} can be dropped.

\begin{corollary} [CSM class of type $A$ quiver orbits]\label{thm:2A}
Let $Q$ be a Dynkin quiver of type A. For the list of dimension vectors $\ddelta'_m=(\underbrace{\beta_1,\ldots, \beta_1}_{m_1}, \underbrace{\beta_2,\ldots, \beta_2}_{m_2}, \ldots, \underbrace{\beta_N,\ldots, \beta_N}_{m_N})$ (cf. (\ref{eqn:lists}))) we have
\[
\csm(\OO_m \subset \Rep_\gamma)=
\frac{1}{\prod_{j=1}^r m_j!}\cdot
\sum_{S \text{ is a } \ddelta'_m\text{-shuffle}}
\fac_1(S)\fac_2(S)\fac_3(S).
\]
\qed
\end{corollary}

\begin{proof} In the cohomology part of Theorem \ref{thm:2} the $c^o$ characteristic class (that is the factor $\fac_4(S)$) only occurs as $c^o_\beta$ for {\em positive roots} $\beta$. In type $A$, for all positive root $\beta$ we have $c^o_\beta=1$, see Example \ref{ex:co1}. Hence in type $A$ the $\fac_4(S)$-factors can be dropped.
\end{proof}

\subsection{The interplay of the different versions} \label{sec:interplay}
In Theorems~\ref{thm:1},~\ref{thm:2} we presented two ways of calculating CSM and MC classes of Dynkin quiver orbits. However, the visual similarity of the two theorems is misleading.

The formula in Theorem \ref{thm:2} is an explicit formula for any Dynkin quiver orbit {\em as long as we know the $c^o_{\beta_j}$, $C^o_{\beta_j}$ classes} for the positive roots $\beta$ of $Q$. For each particular quiver this is finitely many classes to know. More discussion on $c^o_\beta$, $C^o_\beta$ classes see Section \ref{sec:coCo}. 

However, the formula in Theorem \ref{thm:1} has the $c^o_{m_j\beta_j}$, $C^o_{m_j\beta_j}$ classes as ingredients {\em for all non-negative integers $m_j$}. Note that $c^o_{m\beta}$ ($C^o_{m\beta}$) is not an obvious modification of $c^o_\beta$ ($C^o_\beta$), see for example $c^o_{(1,1)}$ and $c^o_{(2,2)}$ in Example \ref{exA2}. So we can hardly consider Theorem \ref{thm:1} alone as an explicit formula for CSM or MC classes of quiver Dynkin orbits. 

A computationally effective algorithm is obtained from the interplay of the two theorems, as follows. One can use the explicit formulas for $c^o_\beta$, $C^o_\beta$ from Section \ref{sec:coCo}, as well as  Theorem \ref{thm:2} for the Kostant partition $(0,\ldots,0,m,0,\ldots,0)$ to calculate $c^o_{m\beta}$'s and $C^o_{m\beta}$'s. Having those at hand makes the calculation in Theorem \ref{thm:1} explicit and faster than using Theorem \ref{thm:2} alone.

\subsection{Examples} \label{sec:exs}

For $Q=(1 \to 2 \to 3)$ consider the dimension vector $(1,2,1)$, and its orbits from Example \ref{ex:A3}, which we will call $\OO_1,\ldots, \OO_5$. Temporarily using the short-hand notation $\alpha_{1,1}=a$, $\alpha_{2,i}=b_i, \alpha_{3,1}=c$ we have
\begin{align*}
\csm(\OO_1)=&
1+(c-a)+\left((c+a)(b_1+b_2)-b_1^2-b_2^2-2ac\right),\\
\csm(\OO_2)=&
(c-a)+\left((c+a)(b_1+b_2)-2ac-2b_1b_2\right),\\
\csm(\OO_3)=&
(c-b_1)(c-b_2)(1+b_1+b_2-2a),\\
\csm(\OO_4)=&
(b_1-a)(b_2-a)(1+2c-b_1-b_2),\\
\csm(\OO_5)=&
(b_1-a)(b_2-a)(c-b_1)(c-b_2).
\end{align*}
The expressions for the classes $\csm(\OO_3), \csm(\OO_4), \csm(\OO_5)$, as well as the fact
\begin{equation}\label{eq:O1O2}
\csm(\OO_1)+\csm(\OO_2)=(1+b_1+b_2-2a)(1+2c-b_1-b_2)
\end{equation}
follow easily from the product property of CSM classes and the calculation of some $c^o$-classes in Example~\ref{exA2}. The novelty here is how the expression in (\ref{eq:O1O2}) is distributed between the CSM classes of $\OO_1$ and $\OO_2$.

In K theory, for the same quiver, dimension vector, and orbits, we have 
\begin{align*}
\mC(\OO_1)=&
(1+y)^2\frac{a}{c}\left(
1+y \clubsuit
-y-\frac{ya}{c}+\frac{y^2a}{c}\right),
\\
\mC(\OO_2)=&
(1+y)^2\frac{a}{c}\left(
1-\clubsuit + y-\frac{ya}{c}+\frac{a}{c}+
\frac{b_1}{b_2}+\frac{b_2}{b_1}\right)
,\\
\mC(\OO_3)=&
(1-\frac{b_1}{c})(1-\frac{b_2}{c})(1+y)
\left(\frac{a}{b_1}+\frac{a}{b_2}-\frac{(1-
y)a^2}{b_1b_2}
\right),\\
\mC(\OO_4)=&
(1-\frac{a}{b_1})(1-\frac{a}{b_2})(1+y)
\left(\frac{b_1}{c}+\frac{b_2}{c}-\frac{(1-
y)b_1b_2}{c^2}
\right)
,\\
\mC(\OO_5)=&
(1-\frac{a}{b_1})(1-\frac{a}{b_2})(1-\frac{b_1}{c})(1-\frac{b_2}{c})
,
\end{align*}
where
\[
\clubsuit=\frac{a}{b_1}+\frac{a}{b_2}+\frac{b_1}{c}+\frac{b_2}{c}.
\]
The expressions for the classes $\mC(\OO_3), \mC(\OO_4), \mC(\OO_5)$, as well as the fact
\begin{equation}\label{eq:O1O2mC}
\mC(\OO_1)+\mC(\OO_2)=(1+y)^2
\left(\frac{a}{b_1}+\frac{a}{b_2}-\frac{(1-
y)a^2}{b_1b_2}
\right) 
\left(\frac{b_1}{c}+\frac{b_2}{c}-\frac{(1-
y)b_1b_2}{c^2}
\right)
\end{equation}
follow easily from the product property of MC classes and the calculation of some MC classes for $A_2$ in Example~\ref{exA2}. The novelty here is how the expression in (\ref{eq:O1O2mC}) is distributed between the MC classes of $\OO_1$ and $\OO_2$.

\section{Cohomological and K-theoretic Hall algebras}

For a quiver $Q$ recall the notation $\HH_\gamma^Q$ and $\KK_\gamma^Q$ from (\ref{def:HHKK}) and define
\[
\HH^Q=\bigoplus_{\gamma\in \N^{Q_0}} \HH^Q_\gamma, \qquad\qquad 
\KK^Q=\bigoplus_{\gamma\in \N^{Q_0}} \KK^Q_\gamma
\]
as vector spaces---in fact we will consider elements in each with infinitly many non-zero $\gamma$-components, that is, $\oplus$ denotes the direct {\em product} above. When $Q$ is clear from the context, we will drop the upper index. 

We need to be careful with notation. Elements of $\HH^Q_\gamma$ ($\KK^Q_\gamma$) for different $\gamma$'s may have the same {\em name}, for example there are ``1''s or even ``$\alpha_{1,1}$''s in multiple $\HH^Q_\gamma$'s ($\KK^Q_\gamma$'s). Therefore, when necessary, we add $\gamma$ as a subscript indicating which $\HH^Q_\gamma$ or $\KK^Q_\gamma$ we mean. For example $1_{(0,0)}\in \HH^{A_2}_{(0,0)}$, $1_{(1,0)}\in \HH^{A_2}_{(1,0)}$, $1_{(2,0)}\in \HH^{A_2}_{(2,0)}$, $1_{(1,1)}\in \HH^{A_2}_{(1,1)}$ are all different elements of $\HH^{A_2}$.

Following \cite{KS}, in \cite{YZ} non-commutative multiplications are introduced on $\HH^Q$, $\KK^Q$. We will give an algebraic definition of these multiplications in Section~\ref{sec:shuffle} and a geometric interpretation in Section \ref{sec:stargeo}.

\subsection{Shuffle multiplication} \label{sec:shuffle}
Let $\ggamma=(\gamma_1,\ldots,\gamma_r)$ be a list of dimension vectors.

\begin{definition} \label{def:shuffleproduct}
For  $f_u\in \HH^Q_{\gamma_u}$, $u=1,\ldots,r$ define
\begin{equation}\label{eqn:multalg}
f_1*\ldots * f_r=\sum_{S \text{ is a $\ggamma$-shuffle}}
\left( \prod_{u=1}^r f_u(S_{*,u})\right) \cdot \fac_1(S)\fac_2(S)\fac_3(S) \qquad \in \HH^Q_\gamma.
\end{equation}
For $f_u\in \KK^Q_{\gamma_u}$, $u=1,\ldots,r$ define
\begin{equation}\label{eqn:multalgK}
f_1*\ldots * f_r=\sum_{S \text{ is a $\ggamma$-shuffle}}
\left( \prod_{u=1}^r f_u(S_{*,u})\right) \cdot \fac'_1(S)\fac'_2(S)\fac'_3(S) \qquad \in \KK^Q_\gamma.
\end{equation}
\end{definition}

The $r=2$ special case of Definition \ref{def:shuffleproduct} defines an associative, non-commutative algebra structure on $\HH^Q$, $\KK^Q$. In fact associativity can be proved by observing that the $*$-product of $f_1,\ldots,f_r$ in this order, but arbitrarily grouped equals the formula in Definition \ref{def:shuffleproduct}. The obtained algebra structure on $\HH^Q$, $\KK^Q$ will be called the Cohomological Hall Algebra (CoHA) and the K-theoretic Hall Algebra (KHA).

\subsection{Geometry of the CoHA/KHA multiplication} \label{sec:stargeo}
Recall the notion of $\Fl_\ggamma$ and the bundles $\SSS_{i,u}, \G^{(1)}_\ggamma$, $\G^{(2)}_\ggamma$, $\G^{(3)}_\ggamma$, $\G^{(4)}_\ggamma$ over $\Fl_\ggamma\times \Rep_\gamma$ from Section~\ref{sec:bundles}. The geometric interpretation of the $*$ multiplication both in CoHA and KHA is
\begin{equation}\label{eqn:multgeo}
f_1*\ldots * f_r=
\pi_{2*}\left(
\left(\prod_{u=1}^r f_u(\SSS_{*,u})\right)
\cdot e(\G^{(1)}_{\ggamma}) c(\G^{(2)}_{\ggamma}) c(\G^{(3)}_{\ggamma}) \right),
\end{equation}
where $\pi_2$ is the projection of $\Fl_{\ggamma} \times \Rep_\gamma$ to the second factor.
Note that the bundles $\SSS_{*,u}$ for a given $u$, is a collection of bundles of ranks $\gamma_u(1),\ldots,\gamma_u(|Q_0|)$. These bundles have exactly as many Chern roots as the variables of $f_u$. Thus the evaluation $f_u(\SSS_{*,u})$ above makes sense.
The map $\pi_2$ is a $\GL_\gamma$-equivariant proper map, the map $\pi_{2*}$ is meant in $\GL_\gamma$-equivariant cohomology (K theory), hence the right hand sides of (\ref{eqn:multgeo}) is indeed an element in $\HH^Q_\gamma$ or $\KK^Q_\gamma$.

The equivariant localization expressions for the push-forward map $\pi_{2*}$ in cohomology or K theory give exactly (\ref{eqn:multalg}) and (\ref{eqn:multalgK}) proving that the algebraic formulas (\ref{eqn:multalg}), (\ref{eqn:multalgK}) for the $*$ multiplication are the same as the geometric formula~(\ref{eqn:multgeo}).

\section{CSM and MC classes of quiver orbits in the CoHA and KHA}

The overlap between Theorem \ref{thm:2} and formulas (\ref{eqn:multalg}), (\ref{eqn:multalgK}), (\ref{eqn:multgeo}) immediately gives the CSM and MC classes of orbits of Dynkin quivers as special elements in the CoHA and KHA.

\begin{theorem}\label{thm:main1}
Let $Q$ be a Dynkin quiver, and let $\beta_1<\ldots<\beta_N$ be the Reineke order of its positive roots. Recall the special elements $c^o_{\beta_j}\in \HH^Q_{\beta_j}$ and $C^o_{\beta_j}\in \KK^Q_{\beta_j}$. Let $m=(m_1, \ldots, m_N)$ be a Kostant partition of the dimension vector $\gamma=\sum_{j=1}^N m_j \beta_j$, and let $\OO_m \subset \Rep_\gamma$ be the corresponding $\GL_\gamma$ orbit.
For the $\GL_\gamma$-equivariant CSM and MC classes we have
\[
\csm(\OO_m\subset \Rep_\gamma)=
\frac{1}{\prod_{u=1}^N m_u!} \cdot
\underbrace{c^o_{\beta_1}*\ldots*c^o_{\beta_1}}_{m_1}*
\ldots
* \underbrace{c^o_{\beta_N}*\ldots*c^o_{\beta_N}}_{m_N},
\]
\[
\mC(\OO_m\subset \Rep_\gamma)=
\frac{1}{\prod_{u=1}^N [m_u]_y!} \cdot
\underbrace{C^o_{\beta_1}*\ldots*C^o_{\beta_1}}_{m_1}*
\ldots
* \underbrace{C^o_{\beta_N}*\ldots*C^o_{\beta_N}}_{m_N}.
\]

\qed
\end{theorem}

\begin{remark}\rm
In type A, in cohomology, the $c^o_{\beta_j}$ classes are all equal to 1, see Example \ref{ex:co1}. Hence the CSM class of type A Dynkin quiver orbits are as simple as $1*1*\ldots *1$ (up to a factorial). Such simplicity holds for the {\em fundamental class} of the quiver orbits in cohomology {\em in any Dynkin type}, see \cite{rrcoha}. The fundamental class is only the lowest degree component of the CSM class. Theorem~\ref{thm:main1} shows that such simplicity needs to be sacrificed for the higher degree terms, already in type $D_4$.
\end{remark}

Even more generally, from the comparison of Proposition \ref{prop:pi2} with the formulas (\ref{eqn:multalg}), (\ref{eqn:multalgK}), (\ref{eqn:multgeo}) we obtain
\begin{align}\label{eqn:genpi2star}
\pi_{2*}(\csm(\Sigma_{\ggamma}\subset \Fl_{\ggamma}\times \Rep_\gamma)) & =c^o_{\gamma_1}*\ldots*c^o_{\gamma_r} \in \HH_\gamma, \\
\pi_{2*}(\mC(\Sigma_{\ggamma}\subset \Fl_{\ggamma}\times \Rep_\gamma)) & =C^o_{\gamma_1}*\ldots*C^o_{\gamma_r} \in \KK_\gamma \notag
\end{align}
for any list $\ggamma=(\gamma_1,\ldots,\gamma_r)$ of dimension vectors with $\sum \gamma_u=\gamma$, which will be important in Section~\ref{sec:proof}.

\section{Donaldson-Thomas type identities for CSM and MC classes}\label{sec:DT}

In this section we re-formulate Theorem \ref{thm:main1} to be an identity between certain products of exponentials in the CoHA (or KHA). Then we will generalize this identity in the style of Donaldson-Thomas invariants.

\subsection{Exponential identities in CoHA, KHA}
Let $\zv$ denote the zero dimension vector. 
\begin{definition} \
\begin{itemize}
\item Let $c\in \HH^Q$ be such that its $\HH^Q_{\zv}$-component is 0. Define 
\[
\Exp(c)=
\sum_{k=0}^\infty \frac{c^{*k}}{k!}=
1_\zv+c+\frac{c*c}{2!}+\frac{c*c*c}{3!}+\ldots \in \HH^Q.
\]
\item Let $C\in \KK^Q$ such that its $\KK^Q_{\zv}$-component is 0. Define 
\[
\Expy(C)=
\sum_{k=0}^\infty \frac{C^{*k}}{[k]_y!}=
1_\zv+C+\frac{C*C}{[2]_y!}+\frac{C*C*C}{[3]_y!}+\ldots \in \KK^Q.
\]
\end{itemize}
\end{definition}

\begin{example} \rm
For $Q$ let $\ep$ be a simple root. We have
\[
\Exp(\underbrace{1}_{\in \HH_{\ep}})=\underbrace{1}_{\in \HH_\zv} +  \underbrace{1}_{\in \HH_{\ep}} + \underbrace{1}_{\in \HH_{2\ep}}+ \underbrace{1}_{\in \HH_{3\ep}}+\ldots,
\]
\[
\Expy(\underbrace{1}_{\in \KK_{\ep}})=\underbrace{1}_{\in \KK_\zv} +  \underbrace{1}_{\in \KK_{\ep}} + \underbrace{1}_{\in \KK_{2\ep}}+ \underbrace{1}_{\in \KK_{3\ep}}+\ldots.
\]
\end{example}

Now we are going to rephrase Theorem \ref{thm:main1} in the language of $\Exp$, $\Expy$. Let $R^+$ denote the set of positive roots, and let $R$ denote the set of simple roots, equivalently, the set $Q_0$ of vertices. An order $<$ of $R$ will be called head-before-tail order if for every $a\in Q_1$ we have $h(a)<t(a)$.

\begin{theorem} \label{thm:coincidence}
In $\HH^Q$, $\KK^Q$ we have
\[
\mathop{\prod_{\beta\in R^+}}_{\text{in Reineke order}} 
   \Exp(c^o_\beta) = 
\mathop{\prod_{\ep\in R}}_{\text{head-before-tail}} \Exp(c^o_{\ep}),
\]
\[
\mathop{\prod_{\beta\in R^+}}_{\text{in Reineke order}} 
   \Expy(C^o_\beta) = 
\mathop{\prod_{\ep\in R}}_{\text{head-before-tail}} \Expy(C^o_{\ep}),
\]
where the products are taken in the indicated orders.
\end{theorem}

\begin{example} \rm
For $A_2=(1\to 2)$ we obtain 
\begin{equation} \label{eqn:CoHAI}
\Exp(1_{(1,0)})\Exp(1_{(1,1)})\Exp(1_{(0,1)})=
\Exp(1_{(0,1)})\Exp(1_{(1,0)}),
\end{equation}
\begin{equation} \label{eqn:KHAI}
\Expy(1_{(1,0)})\Expy\left(\frac{(1+y)\alpha_{1,1}}{\alpha_{2,1}}_{(1,1)}\right)\Expy(1_{(0,1)})=
\Expy(1_{(0,1)}) \Expy(1_{(1,0)}).
\end{equation}
The $\HH^{A_n}$ statement is analoguous to (\ref{eqn:CoHAI}):  on both sides we have a product of exponentials of $1$'s ($\binom{n+1}{2}$ factors on the left, and $n$ factors on the right), because in type $A$ all $c^o_\beta=1$.  However, already in type $D_4$ there is a $c^o_\beta$ not 1, see Example~\ref{D4D5E6}, hence already the $\HH^{D_4}$ statement has more interesting left hand side.
\end{example}

\begin{proof}
We have
\[
\mathop{\prod_{\beta\in R^+}}_{\text{in Reineke order}} 
   \Exp(c^o_\beta) = \sum_{m\in \N^{R^+}} \csm(\OO_m)=
\sum_{\gamma\in \N^{Q_0}} \csm(\Rep_\gamma)=
\mathop{\prod_{\ep\in R}}_{\text{head-before-tail}} \Exp(c^o_{\ep}),
\]
\[
\mathop{\prod_{\beta\in R^+}}_{\text{in Reineke order}} 
   \Expy(C^o_\beta) = \sum_{m\in \N^{R^+}} \mC(\OO_m)=
\sum_{\gamma\in \N^{Q_0}} \mC(\Rep_\gamma)=
\mathop{\prod_{\ep\in R}}_{\text{head-before-tail}} \Expy(C^o_{\ep}),
\]
where we used the shorthand notations $\csm(\OO_m)=\csm(\OO_m\subset \Rep_\gamma)$ for $\gamma=\sum m_\beta\beta$ (and similar for $\mC$) as well as $\csm(\Rep_\gamma)=\csm(\Rep_\gamma\subset \Rep_\gamma)$ (and similar for $\mC$). The first equalities in both lines are reformulations of Theorem \ref{thm:main1}. The second equality in both lines follow from the additive property of CSM and MC classes. The third equality in both lines follow from the explicit shuffle form of the * multiplication. 
\end{proof}

As can be seen from the proof, the identities in Theorem \ref{thm:coincidence} encode the geometric facts that the CSM (MC) classes of the orbits of $\Rep_\gamma$ add up to the CSM (MC) class of $\Rep_\gamma$. These are truly remarkable identities between rational functions, for example,
\begin{align}\label{eqn:ident}
\sum_{i=1}^5 \csm(\OO_i)=&
(1+b_1-a)(1+b_2-a)(1+c-b_1)(1+c-b_2),\\
\sum_{i=1}^5 \mC(\OO_i)=&
(1+y\frac{a}{b_1})(1+y\frac{a}{b_2})(1+y\frac{b_1}{c})(1+y\frac{b_2}{c}) \notag
\end{align}
for the functions of Section \ref{sec:exs}. Even more remarkable is that Theorem \ref{thm:coincidence} encodes these identities {\em for every dimension vector at the same time}.  

Theorem \ref{thm:coincidence} is a special case of a more general one that we will prove in the next section.

\subsection{Identities parameterized by stability functions}

Let $\A_Q$ be the set of isomorphism classes of $\C Q$-representations, equivalently, the set of orbits of $\Rep_\gamma$'s for all dimension vector $\gamma$. Denote the additive semigroup $\N^{Q_0}$ of dimension vectors by $K_0(\A_Q)$. The dimension vector map $\A_Q\to K_0(\A_Q)$ will be denoted by $\dim$. 

A stability function (aka. central charge) $Z$ is an additive homomorphism $K_0(\A_Q)\to \R^2$ such that only $\zv$ maps to $0$, and the {\em phase} ($\arctan(y/x)$, the angle measured counterclockwise from the positive $x$-axis) of the $Z$-image of any non-zero vector is in $[0,\pi)$. The map $Z$ is determined by its restriction to the simple roots $\ep_i$. We will only consider generic stability functions, that is, we assume that $Z$ maps two non-zero dimension vectors into the same straight line in $\R^2$ only if they are $\Q$-proportional. 

For a given $Z$ an element $M\in \A_Q$ is called {\em semistable} ({\em stable}) if for each non-zero proper subobject $Y$ of $M$ the phase of $\dim(Y)$ is less than or equal to the phase of $\dim(M)$ (strictly less than the phase of $\dim(M)$). Only indecomposable modules have a chance to be stable, and only integer multiples of indecomposables have a chance to be semistable.

For a Dynkin quiver $Q$ and general stability function $Z$ consider the products 
\begin{equation}\label{eqn:genprod}
\mathop{\prod^\curvearrowright_{\beta\in R^+}}_{M_\beta\text{ is stable for }Z} 
\Exp(c^o_\beta),
\qquad\qquad
\mathop{\prod^\curvearrowright_{\beta\in R^+}}_{M_\beta\text{ is stable for }Z} 
\Expy(C^o_\beta)
\end{equation}
in $\HH^Q$ and $\KK^Q$ obtained by taking the product in the order of {\em decreasing phase}---this order convention is indicated by the $\curvearrowright$ symbol.

\begin{theorem}\label{thm:Z}
The elements (\ref{eqn:genprod}) of $\HH^Q$, $\KK^Q$ are independent of the general stability function~$Z$.
\end{theorem}

\begin{example} \rm \label{ex:ZA2}
For $Q=A_2=(1\to 2)$ there are three positive roots: $\ep_1$, $\ep_2$ and $\ep_1+\ep_2$. We have $M_{\ep_2}\subset M_{\ep_1+\ep_2}$. There are two generic stability functions combinatorially: either phase$(Z(\ep_1))>$ phase$(Z(\ep_2))$, or phase$(Z(\ep_2))>$ phase$(Z(\ep_1))$. In the first case there are three stable objects $M_{\ep_1}$, $M_{\ep_2}$, $M_{\ep_1+\ep_2}$. In the second case there are only two stable objects $M_{\ep_1}, M_{\ep_2}$. In both cases the semistable objects are the multiples of the stable ones. The statement of Theorem \ref{thm:Z} for $A_2$ is thus equivalent to (\ref{eqn:CoHAI}), (\ref{eqn:KHAI}). 
\end{example}

\begin{remark} \label{ex:qdilog}
Define the quantum dilogarithm function (formal power series) by
\[
\EE(z)=\sum_{j=0}^\infty \frac{ (-z)^jq^{j^2/2} }{\prod_{k=1}^j(1-q^k)}.
\]
In the theory of Donaldson-Thomas invariants of quivers a statement similar to our Theorem~\ref{thm:Z} is proved (see \cite{KS, keller, rrcoha, AR}), namely, that the product
\[
\mathop{\prod^\curvearrowright_{\beta\in R^+}}_{M_\beta\text{ is stable for }Z} 
\EE(y_\beta)
\]
is independent of the stability function $Z$. Here, the products need to be evaluated in a {\em quantum algebra} (not defined here) instead of $\HH^Q$, and $y_\beta$ are special elements in that algebra. 
Both this theorem and Theorem \ref{thm:Z} boil down to {\em ordinary} (commutative) identities among rational functions for every dimension vector. To illustrate the nature of them, here is the example for $Q=(\circ \to \circ \to \circ)$ and dimension vector $(1,2,1)$: Theorem \ref{thm:Z} boils down to the identities (\ref{eqn:ident}), while the quantum-dilogarithm identity boils down to the identity
\[
\frac{1}{(1-q)^3(1-q^2)}=
\frac{1}{(1-q)^2}+
\frac{q}{(1-q)^2}+
\frac{q^2}{(1-q)^3}+
\frac{q^2}{(1-q)^3}+
\frac{q^4}{(1-q)^3(1-q^2)}.
\]
As can be seen in this example, the identities induced by Theorem \ref{thm:Z} depend on many more variables than the one induced by the quantum dilogarithm identities. We are, however, not aware of any specialization or other similar operation applied to our identities to recover the quantum dilogarithm identities.
\end{remark}

\subsection{Proof of Theorem \ref{thm:Z}} \label{sec:proof}

Let $Z$ be a generic stability function for $Q$. Let $\beta_1,\ldots, \beta_r$ be the list of those positive roots for which $M_\beta$ is stable for $Z$, in the order of decreasing $Z$-phase. A Kostant partition $(m_\beta)_{\beta\in R^+}$ will be called $Z$-compatible, if $m_\beta\not=0$ only for $\beta_1,\ldots,\beta_r$, that is, such a Kostant partition is $(m_1,\ldots,m_r)$ where we use the short hand notation $m_i=m_{\beta_i}$.

Let us fix a dimension vector $\gamma\in \N^{Q_0}$. For a $Z$-compatible Kostant partition for $\gamma$, consider the list of dimension vectors 
\[
\ddelta_m=(m_1\beta_1, m_2\beta_2,\ldots,m_r\beta_r).
\]
We have the incidence variety $\Sigma_{\ddelta_m}\subset \Fl_{\ddelta_m}\times \Rep_\gamma$, and let us denote its projection to $\Rep_\gamma$ by~$\pi_m$. 

We claim that the collection of maps $\pi_m: \Sigma_{\ddelta_m}\to \Rep_\gamma$ for all $Z$-compatible Kostant partitions $m$ for $\gamma$ have the following properties:
\begin{itemize}
\item each $\pi_m$ is one-to-one (it is an isomorphism to its image),
\item the images $\pi_m(\Sigma_{\ddelta_m})$ are disjoint, and 
\item their union is $\Rep_\gamma$.
\end{itemize}
For an element in $\Rep_\gamma$ let $M$ be the corresponding $\CQ$-module. As we saw in Section \ref{sec:filt}, a preimage of $M$ at any of the $\pi_m$'s ($m$ is a $Z$-compatible Konstant partition for $\gamma$) is a filtration 
\[
0=M^{(0)}\subset M^{(1)} \subset \ldots \subset M^{(r)}=M
\]
of $\CQ$-modules, with subquotients corresponding to the open orbits in $\Rep_{m_1\beta_1},\ldots, \Rep_{m_r\beta_r}$ (in this order). The orbit corresponding to $m_uM_{\beta_u}$ has codimension $\dim \Ext(mM_\beta,mM_\beta)=m^2 \Ext(M_\beta,M_\beta)=0$ (cf. \cite[Cor. 2.6]{Ki}), hence the subquotients of the filtration above are $m_1M_{\beta_1},\ldots,m_rM_{\beta_r}$ (in this order).
Thus our claim above is proved by the following. 

\begin{theorem}[\cite{King}]
Let $Z$ be a stability function, and $M$ a $\CQ$-module. Then $M$ admits a unique filtration (called the Harder-Narasimhan filtration) of $\CQ$-modules
\[
0=M^{(0)}\subset M^{(1)} \subset \ldots \subset M^{(r)}=M
\]  
whose subquotients $M^{(u)}/M^{(u-1)}$ are semistable for $Z$, with strictly decreasing phases.  \qed
\end{theorem}

From the claim above, the additivity of CSM classes yield
\begin{equation}\label{eqn:tt}
\csm(\Rep_\gamma\subset \Rep_\gamma)=
\mathop{\sum_{m \text{ is $Z-$}}}_{\text{compatible for }\gamma} \pi_{2*}(\csm(\Sigma_{\ddelta_m}\subset \Rep_\gamma)).
\end{equation}
For the list of dimension vectors 
$
\ddelta'_m=
(\underbrace{\beta_1,\ldots, \beta_1}_{m_1},
\underbrace{\beta_2,\ldots, \beta_2}_{m_2},
\ldots,
\underbrace{\beta_r,\ldots, \beta_r}_{m_r})
$
we have
\begin{align*}
\pi_{2*}(\csm(\Sigma_{\ddelta_m}\subset \Rep_\gamma))=&\frac{1}{\prod_{u=1}^r m_u!} \pi_{2*}(\csm(\Sigma_{\ddelta'_m}\subset \Rep_\gamma))
\\
=&\frac{1}{\prod_{u=1}^r m_u!} \cdot
\underbrace{c^o_{\beta_1}*\ldots*c^o_{\beta_1}}_{m_1}*
\ldots
* \underbrace{c^o_{\beta_r}*\ldots*c^o_{\beta_r}}_{m_r},
\end{align*}
where, in the first equality we used  Lemma \ref{lem:fibr}---as well as the fact that $\chi(\Fl(m))=m!$---and in the second equality we used  (\ref{eqn:genpi2star}).
Plugging into (\ref{eqn:tt}) we obtain 
\begin{equation}\label{eqn:gammafix}
\csm(\Rep_\gamma\subset \Rep_\gamma)=
\mathop{\sum_{m \text{ is $Z-$}}}_{\text{compatible for }\gamma} \frac{1}{\prod_{u=1}^r m_u!} \cdot
\underbrace{c^o_{\beta_1}*\ldots*c^o_{\beta_1}}_{m_1}*
\ldots
* \underbrace{c^o_{\beta_r}*\ldots*c^o_{\beta_r}}_{m_r}.
\end{equation}
Adding these equations together for all $\gamma\in \N^{Q_0}$ we obtain 
\[
\mathop{\prod^\curvearrowright_{\beta\in R^+}}_{M_\beta\text{ is stable for }Z} 
\Exp(c^o_\beta)
\]
on the right hand side, and an expression independent of $Z$ on the left hand side. This proves the theorem for $\HH^Q$. The $\KK^Q$ case is obtained the same way, replacing $\csm$ with $\mC$, $c^o$ with $C^o$ and $m_u!$ (the Euler charactersitic of $\Fl(m_u)$) with $[m_u]_y!$ (the $\chi_y$-genus of $\Fl(m_u)$). \qed

\section{The $c^o_\beta$ and $C^o_\beta$ classes} \label{sec:coCo}

According to Theorems \ref{thm:1} and \ref{thm:2} as well as Theorem \ref{thm:Z} the classes $c^o_\beta$, $C^o_\beta$ are the main building blocks of CSM/MC theory of Dynkin quiver orbits, as well as the CoHA and KHA. In this section we study how to calculate these classes. First we present the obvious case.

\begin{theorem}\label{thm:co1}
Let $Q$ be Dynkin quiver and $\beta$ a positive root with all coordinates $\leq 1$ (this holds for all positive roots in type $A$, and for some in other types). We have
\[
c^o_\beta=1, \qquad C^o_\beta=\mathop{\prod_{a\in Q_1}}_{\beta(t(a))=\beta(h(a))=1} (1+y)\frac{\alpha_{t(a),1}}{\alpha_{h(a),1}}.
\]
\end{theorem}

\begin{proof}
For each $a\in Q_1$ with $\beta(t(a))=\beta(h(a))=1$ we have
\begin{align*}
\csm\left( \Hom(\C^{\beta(t(a))},\C^{\beta(h(a))}) - \{0\} \subset \Hom(\C^{\beta(t(a))},\C^{\beta(h(a))})\right) =&1, \\
\mC\left( \Hom(\C^{\beta(t(a))},\C^{\beta(h(a))}) - \{0\} \subset \Hom(\C^{\beta(t(a))},\C^{\beta(h(a))}) \right) =&(1+y)\frac{\alpha_{t(a),1}}{\alpha_{h(a),1}},
\end{align*}
due to Examples \ref{ex:coordinateplanes}, \ref{ex:coordinateplanes_mC}.
The open orbit in $\Rep_\beta$ is the product of the set of injective maps (ie. $\Hom(\C^{\beta(t(a))},\C^{\beta(h(a))}) - \{0\}$) for all $a\in Q_1$ with $\beta(t(a))=\beta(h(a))=1$.
\end{proof}

In Section \ref{sec:co} we gave some examples for $c^o_\beta$, $C^o_\beta$ classes for positive roots $\beta$ {\em not} satisfying the assumptions of Theorem \ref{thm:co1}.
They can be proved by any of the three methods we will show in Sections \ref{sec:sieve} (sieve method), \ref{sec:interpol} (interpolation method), \ref{sec:commutator} (commutator method). Using those methods we calculated many other classes. The $C^o_\beta$ classes tend to be long, complicated formulas. However, the {\em many} $c^o_\beta$ examples we calculated support the following conjecture.

\begin{conjecture} \label{con:key}
For any positive root $\beta$ we have
\[
c^o_\beta=\prod_{i\in Q_0} \prod_{u=1}^{\beta(i)} \prod_{v=1}^{\beta(i)} (1+\alpha_{i,u}-\alpha_{i,v})
\]
\end{conjecture}

\noindent It is instructive to verify how this conjecture recovers the $c^o_\beta$ classes of Theorem \ref{thm:co1} and Example~\ref{D4D5E6}. 

\begin{remark} \rm 
The simplicity of the formula  in the conjecture---an equivariant total Chern classes of the form $c(\oplus_i \Hom(V_i,V_i))$---suggests that there is a simple geometric reason for it. Yet, we should mention that the corresponding K theory formulas are much more complicated: they do not factor to linear factors, and they depend on more variables. 
\end{remark}

In the next three subsections we will describe three approaches to calculate $c^o_\beta$, $C^o_\beta$ classes. The first two use ``brute force'' and hence are of limited applicability. The second one is based on the relation of motivic characteristic classes to Okounkov's stable envelopes. The third one uses the sophisticated algebra structure of $\HH^Q$, $\KK^Q$. 

\subsection{Sieve methods} \label{sec:sieve}

Consider the partial order on $R^+$ defined by $\beta_1\leq \beta_2$ if $\beta_1(i)\leq \beta_2(i)$ for all $i$. Theorem \ref{thm:co1} names $c^o_\beta$, $C^o_\beta$ for some $\leq$-small positive roots. Suppose we want to calculate $c^o_\beta$ ($C^o_\beta$) when we already know all $c^o_{\beta'}$ ($C^o_{\beta'}$) for all $\beta'\leq \beta$. Consider the representation $\Rep_\beta$. The sum of the CSM (MC) classes of all orbits equals the CSM (MC) class of $\Rep_\beta$, that is,
\begin{align} \label{whole}
\csm(\Rep_\beta\subset \Rep_\beta)=&
\prod_{a\in Q_1} \prod_{u=1}^{\beta(t(a))} \prod_{v=1}^{\beta(h(a))} (1+\alpha_{h(a),v}-\alpha_{t(a),u}),
\\
\mC(\Rep_\beta\subset \Rep_\beta)=&
\prod_{a\in Q_1} \prod_{u=1}^{\beta(t(a))} \prod_{v=1}^{\beta(h(a))} \left(1+h\frac{\alpha_{t(a),u}}{\alpha_{h(a),v}}\right). \notag
\end{align}
The CSM (MC) orbits of all the non-open orbits are determined by the $c^o_{\beta'}$ ($C^o_{\beta'}$) classes for $\beta'\leq \beta$ via Theorem \ref{thm:2}. Thus, by subtraction, we obtain the CSM/MC classes of the open orbits: $c^o_\beta$, $C^o_\beta$. 

\begin{example} \rm \label{124}
For $Q=D_4$ of Example \ref{D4D5E6} and $\beta=(1,1,2,1)$ there are 15 orbits of $\Rep_\beta$. Thus $c^o_{\beta}$ ($C^o_{\beta}$) is obtained by subtracting 14 rather complicated---but due to  Theorem \ref{thm:2} expli\-cit--- po\-ly\-no\-mi\-als (Laurent polynomials) from the explicit formula (\ref{whole}). Doing so, in cohomology, a massive cancellation is witnessed and the difference is $(1+\alpha_{3,1}-\alpha_{3,2})(1+\alpha_{3,2}-\alpha_{3,1})$. In K theory the cancellation is less massive, and we get the formula of Example \ref{D4D5E6}.
\end{example}

An improvement of the described method is the following. Let $\Inj_\beta\subset \Rep_\beta$ contain tuples of linear maps which are all injective. The open orbit of $\Rep_\beta$ belongs to $\Inj_\beta$. For every $a\in Q_1$ the class
\[
\csm\left( \{\C^{\beta(t(a))}\to \C^{\beta(h(a))}, \text{injective}\} \subset \Hom(\C^{\beta(t(a))}, \C^{\beta(h(a))}\right)
\]
(or MC of the same) can be calculated using Theorem \ref{thm:2} for $Q=A_2$. Then the CSM (MC) class of $\Inj_\beta \subset \Rep_\beta$ is the product of these classes for all $a\in Q_1$. Thus the same sieve argument as above works, by replacing $\Rep_\beta$ with $\Inj_\beta$.

\begin{example} \label{ex:D4sieve} \rm
Let $Q=D_4$ as in Example \ref{D4D5E6}. In $\Rep_{(1,1,2,1)}$ only 4 of the 15 non-open orbits are in $\Inj_{(1,1,2,1)}$ (cf. Example \ref{124}). Hence, if we calculate $c^o_{(1,1,2,1)}$ or $C^o_{(1,1,2,1)}$ for $Q=D_4$ using the described improvement, then we only need to work with the CSM/MC class of 4 of the 12 orbits, and we obtain 
\begin{multline*}
c^o_{(1,1,2,1)}=
\prod_{x\in\{a,b,c\}} (1+d_1+d_2-2x)-
\Delta\left( 
(1+d_1-a)(d_2-b)(d_2-c)+\right.\\
(1+d_1-b)(d_2-a)(d_2-c)+
(1+d_1-c)(d_2-a)(d_2-b)+\\
\left. (d_2-a)(d_2-b)(d_2-c)\right)
=(1+d_1-d_2)(1+d_2-d_1),
\end{multline*}
where we used the temporary notation $\alpha_{1,1}=a, \alpha_{2,1}=b, \alpha_{4,1}=c, \alpha_{3,i}=d_i$, as well as the operator
\[
\Delta(f(d_1,d_2))=f(d_1,d_2)\frac{1+d_2-d_1}{d_2-d_1}+f(d_2,d_1)\frac{1+d_1-d_2}{d_1-d_2}.
\]
\end{example}

\subsection{Interpolation method}\label{sec:interpol}
In this section we briefly recall the interpolation characterization of CSM and MC classes that were proved in \cite{RV, FRcsm} and \cite{FRW} respectively, and are strongly motivated by works of Okounkov and his coauthors \cite{MO1, O, AO1, AO2} on stable envelopes.

Let $V$ be an algebraic representation of the group $G$ with finitely many orbits. For an orbit $\Omega$ let $T_\Omega$ and $N_\Omega$ be the tangent and normal spaces of the orbit at a point $x_\Omega\in \Omega$. These spaces are representations of $G_\Omega$, the (maximal compact subgroup of the) stabilizer group of $x_\Omega$, and let $\T_\Omega$ be the maximal torus of $G_\Omega$. One obtains natural maps $\phi_\Omega:H^*(BG)\to H^*(BG_\Omega)$ and $\phi^K_\Omega:K_G(V)\to R(\T_\Omega)$ ($R$ stands for representation ring, see ibidem for details).

Let $\Ne(f)$ be the Newton polygon of a Laurent polynomial in several variables, and let us use the short hand notations $\csm(\Omega)=\csm(\Omega\subset V)$, $\mC(\Omega)=\mC(\Omega\subset V)$.

\begin{theorem}[\cite{FRcsm} Section 2.5] \label{thm:Hax}
Suppose the representation contains the scalars and $e(N_\Omega)$ is not 0 in $H^*(BG_\Omega)$ for all $\Omega$. Then the axioms
\begin{itemize}
\item $\phi_\Omega(\csm(\Omega))=c(T_\Omega)e(N_\Omega)$,
\item $\phi_{\Theta}(\csm(\Omega))$ is divisible by $c(T_\Theta)$,
\item $\deg(\phi_{\Theta}(\csm(\Omega)))< \deg(c(T_\Theta)e(N_\Theta))$  for $\Theta\not=\Omega$
\end{itemize}
uniquely characterize $\csm(\Omega)\in H_G(V)$. \qed
\end{theorem}

\begin{theorem}[\cite{FRW} Section 5] \label{thm:Kax}
Suppose the representation contains the scalars, the groups $G_\Omega$ are connected, the representations of $\T_\Omega$ on $N_\Omega$ are positive (the weights are contained in a half-space). Then the axioms
\begin{itemize}
\item $\phi^K_\Omega(\mC(\Omega))=c(T_\Omega)e(N_\Omega)$,
\item $\phi^K_{\Theta}(\mC(\Omega))$ is divisible by $c(T_\Theta)$,
\item $\Ne( \phi^K_\Theta(\mC(\Omega))/c(T_\Theta)  )\subset \Ne( e(N_\Theta) )-\{0\}$  for $\Theta\not=\Omega$
\end{itemize}
uniquely characterize $\mC(\Omega)\in K_G(V)$. \qed
\end{theorem}

For Dynkin quiver representations the conditions of Theorems \ref{thm:Hax} and \ref{thm:Kax} are satisfied and all the ingredients ($G_\Omega$, $\T_\Omega$, $T_\Omega$, $N_\Omega$, $\phi_\Omega$)  are explicitly calculated in \cite{tegez}. Therefore, for these representations Theorems \ref{thm:Hax} and \ref{thm:Kax} reduce the calculation of CSM and MC classes of orbits to (huge) systems of {\em linear} equations. Hence, in theory, one can program these linear equations to a computer for $\Rep_\beta$ and obtain $c^o_\beta$ and $C^o_\beta$. We did this for the non-trivial classes of $D_4$ and $D_5$ of Example \ref{D4D5E6}, but for larger examples we find this interpolation method less effective than the one in the next subsection.

\subsection{Simple wall crossing, commutators in CoHA, KHA}\label{sec:commutator}

Let $\omega$, $\beta$, $\tau$ be positive roots for $Q$ and let $\beta=\omega+\tau$. Suppose $Z_1$ and $Z_2$ are stability functions for $Q$ such that 
\begin{itemize}
\item in the order of positive roots with decreasing $Z_1$-phase $(\omega, \beta, \tau)$ is an interval;
\item the order of positive roots with decreasing $Z_2$-phase is the same as that for $Z_1$, except $(\omega, \beta, \tau)$ is replaced with $(\tau, \beta, \omega)$,
\end{itemize}
 as in the picture below.
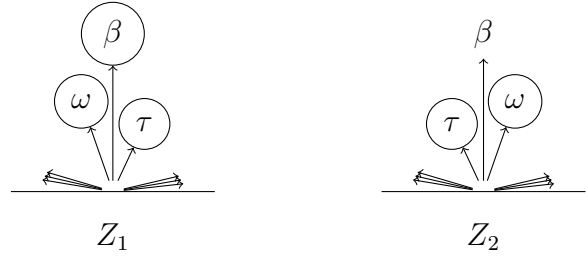
\[
\begin{tikzpicture}[scale=0.6]
\node (A) at (-.5,0) {};
\node (B) at (2,0) {};
\node (C) at (4.5,0) {};
\node (D) at (0.2,.3) {};
\node (E) at (0.25,.4) {};
\node (F) at (0.3,.5) {};
\node (G) at (1.3,2) [draw,circle] {$\omega$};
\node (H) at (2,3.5) [draw,circle]{$\beta$};
\node (I) at (2.7,1.5)[draw,circle] {$\tau$};
\node (J) at (3.7,.4) {};
\node (K) at (3.75,.3) {};
\node (L) at (3.8,.2) {};
\draw [-] (A) to (C);
\draw [->] (B) to (D);
\draw [->] (B) to (E);
\draw [->] (B) to (F);
\draw [->] (B) to (G);
\draw [->] (B) to (H);
\draw [->] (B) to (I);
\draw [->] (B) to (J);
\draw [->] (B) to (K);
\draw [->] (B) to (L);
\node (M) at (2,-1) {$Z_1$};
\end{tikzpicture}
\qquad\qquad
\begin{tikzpicture}[scale=0.6]
\node (A) at (-.5,0) {};
\node (B) at (2,0) {};
\node (C) at (4.5,0) {};
\node (D) at (0.2,.3) {};
\node (E) at (0.25,.4) {};
\node (F) at (0.3,.5) {};
\node (G) at (2.7,2) [draw,circle]{$\omega$};
\node (H) at (2,3.5) {$\beta$};
\node (I) at (1.3,1.5)[draw,circle] {$\tau$};
\node (J) at (3.7,.4) {};
\node (K) at (3.75,.3) {};
\node (L) at (3.8,.2) {};
\draw [-] (A) to (C);
\draw [->] (B) to (D);
\draw [->] (B) to (E);
\draw [->] (B) to (F);
\draw [->] (B) to (G);
\draw [->] (B) to (H);
\draw [->] (B) to (I);
\draw [->] (B) to (J);
\draw [->] (B) to (K);
\draw [->] (B) to (L);
\node (M) at (2,-1) {$Z_2$};
\end{tikzpicture}
\]
One may think that a path in the space of stability functions, connecting $Z_1$ with $Z_2$, crosses over a ``wall'' containing a non-generic stability function where all $\omega$, $\tau$, $\beta$ have the same phase.  

Assume that $M_\omega, M_\beta,M_\tau$ are all stable for $Z_1$, and assume that $M_\tau\subset M_\beta$. It follows that while $M_\tau$ and $M_\omega$ are also stable for $Z_2$, the module $M_\beta$ is not $Z_2$-stable, since the phase of $Z_2(\tau)$ is not less than $Z_2(\beta)$. (In the picture above we indicated stability with circles.)

\begin{remark}\rm 
The prototype example for the described situation is $A_2=(1\to 2)$, $\tau=\ep_2$, $\omega=\ep_1$, $\beta=\ep_1+\ep_2$ (cf. Example \ref{ex:ZA2}). A large example for the $E_8$ quiver (with edges all oriented towards the degree 3 vertex) is $\tau=(1,2,3,2,2,1,1,2)$, $\omega=(1,2,3,3,2,2,1,1)$, $\beta=\tau+\omega$. Here we used the numbering of vertices being a chain from 1 to 7 with another edge connecting 8 and 3. For a given $\beta$ the choice of the other two positive roots is not unique. For example, for the $E_6$ quiver of Example \ref{D4D5E6} for $\beta=(1,2,3,2,1,2)$ we may choose $\tau_1=(1,1,2,1,1,1)$ or $\tau_2=(0,1,2,2,1,1)$.
\end{remark}

\begin{theorem} \label{thm:commutator}
Under the assumptions above we have $c^o_\beta$ and $C^o_\beta$ as $*$-commutators 
\[
c^o_\beta=[c^o_{\tau},c^o_\omega],\qquad\qquad\qquad C^o_\beta=[C^o_{\tau},C^o_\omega]
\]
in the CoHA and the KHA.
\end{theorem}

\begin{proof} 
Equation (\ref{eqn:gammafix}) depends on a dimension vector $\gamma$ and a stability function $Z$. Consider two instances of this equation: one with $\gamma=\beta$ and $Z=Z_1$, and the other with $\gamma=\beta$, and $Z=Z_2$.  The left hand sides of these two equations are the same (the left hand side of (\ref{eqn:gammafix}) does not depend on $Z$). Hence the right hand sides are also the same. However, since the decreasing phase order of positive roots for $Z_1$ and $Z_2$ only differ in the order of $\omega$, $\beta$, $\tau$, the right hand sides also have the same terms, except the terms $c^o_{\omega}*c^o_\tau$, $c^o_\beta$ only appear in the first equation, and the term $c^o_{\tau}*c^o_{\omega}$ only appears in the second equation (and similar for MC). We obtain
\begin{align*}
c^o_{\omega}*c^o_\tau+c^o_\beta = & c^o_{\tau}*c^o_{\omega} \qquad \in \HH^Q, \\
C^o_{\omega}*C^o_\tau+C^o_\beta = & C^o_{\tau}*C^o_{\omega} \qquad \in \KK^Q,
\end{align*}
which, after rearrangement, proves the theorem.
\end{proof}

The author's experience shows that Theorem \ref{thm:commutator} is a more effective method to calculate $c^o_\beta$, $C^o_\beta$ characteristic classes than the sieve and interpolation methods of the preceding two sections. 

\begin{example}\rm
For $A_2=(1 \to 2)$ the positive roots $\tau=(0,1)$, $\omega=(1,0)$ satisfy the assumptions above (cf. Example \ref{ex:ZA2}), hence in $\HH^{A_2}$ and $\KK^{A_2}$ we have
\begin{align*}
c^o_{(1,1)}=&\left[1_{(0,1)},1_{(1,0)}\right]=(1+\alpha_{2,1}-\alpha_{1,1})-(\alpha_{2,1}-\alpha_{1,1})=1,
\\
C^o_{(1,1)}=&
\left[1_{(0,1)},1_{(1,0)}\right]=
\left(1+\frac{y\alpha_{1,1}}{\alpha_{2,1}}\right)-
\left(1-\frac{\alpha_{1,1}}{\alpha_{2,1}}\right)=
(1+y)\frac{\alpha_{1,1}}{\alpha_{2,1}}.
\end{align*}
Of course, these classes we knew already. Note that these calculations essentially coincide with the ``fundamental calculation'' of \cite[Section 2.7]{FRW}. 
\end{example}

\begin{example} \rm
For $Q=D_4$ as in Example \ref{D4D5E6}, set $\tau=(1,0,1,0)$, $\omega=(0,1,1,1)$ and $\beta=(1,1,2,1)$.  In $\HH^{D_4}$ and $\KK^{D_4}$ let us use the temporary notation of Example \ref{ex:D4sieve}. 
We have
\begin{multline*}
c^o_\beta=\left[1_{\tau},1_{\omega}\right]=
\Delta\left( (1+d_1-b)(1+d_1-c)(d_2-a)- \right.
\\
\left. (1+d_1-a)(d_2-b)(d_2-c)\right)=
(1+d_1-d_2)(1+d_2-d_1),
\end{multline*}
\begin{multline}\label{izeize} 
C^o_\beta=\left[  (1+y)\frac{a}{d_1},(1+y)^2\frac{bc}{d_1^2}\right]=
(1+y)^3\Delta'\left(   
\frac{abc}{d_1d_2^2}(1-\frac{a}{d_2})(1+\frac{yb}{d_1})(1+\frac{yc}{d_1})  \right.
\\
\left.
-\frac{abc}{d_1^2d_2}(1+\frac{ya}{d_1})(1-\frac{b}{d_2})(1-\frac{c}{d_2}) \right),
\end{multline}
where we used the operator
\[
\Delta'(f(d_1,d_2))=f(d_1,d_2)\frac{1+yd_1/d_2}{1-d_1/d_2}+f(d_2,d_1)\frac{1+yd_2/d_1}{1-d_2/d_1}.
\]
Carrying out the calculation in (\ref{izeize}) proves the formula we already presented in Example \ref{D4D5E6}.
\end{example}

\section{Future directions}

\subsection{Elliptic characteristic classes}
In this paper we discussed the relationship between motivic characteristic classes, Hall algebras, and DT-type identities in two cohomology theories: ordinary cohomology and K theory. The main ingredients---characteristic classes \cite{BL}, Hall algebras \cite{YZ}, stable envelopes \cite{AO1, FRV, RTV3}---exist on the natural ``third level'': elliptic cohomology. Since the notion of elliptic characteristic classes is not motivic, an elliptic generalization of our results presents new challenges. 

\subsection{Keller-type generalizations for quivers with potentials}
As shown in Remark \ref{ex:qdilog} our Theorem \ref{thm:Z} has the ``shape'' of quantum dilogarithm identities associated with Dynkin quivers. Those q-dilogarithm identities can be generalized to certain non-Dynkin quivers, see \cite{keller, KS, AR}.  Those generalizations are naturally associated with quivers {\em with potentials}, and the cohomology theories are accordingly modified to take into account the potential (cf. rapid decay cohomology). It would be interesting to study whether our Theorem \ref{thm:Z} has a similar generalization.


\begin{thebibliography}{[AMSS2]}

\raggedbottom

\bibitem[AO1]{AO1}
M. Aganagic, A. Okounkov, {Elliptic stable envelopes}, preprint 2016, arXiv:1604.00423

\bibitem[AO2]{AO2}
M. Aganagic, A. Okounkov, {Quasimap counts and Bethe eigenfunc\-tions}, ar\-Xiv:1704.\-08746

\bibitem[AR]{AR} 
J. Allman, R. Rimanyi: Quantum dilogarithm identities for the square product of A-type Dynkin quivers; to appear in Mathematical Research Letters, 2017 

\bibitem[AM1]{AM1} P. Aluffi, L. C. Mihalcea: Chern classes of Schubert cells and varieties; J.
Algebraic Geom., 18(1):63--100, 2009.

\bibitem[AM2]{AM2} P. Aluffi, L. C. Mihalcea: Chern-Schwartz-MacPherson classes for Schubert cells in flag manifolds, Compositio Mathematica, Vol. 152, Issue 12, 2016, pp.
    2603--2625

\bibitem[AMSS1]{AMSS} P. Aluffi, L. C. Mihalcea, J. Sch\"urmann, C. Su.
Shadows of characteristic cycles, Verma modules, and positivity of Chern-Schwartz-MacPherson classes of Schubert cells,  arXiv:1709.08697

\bibitem[AMSS2]{AMSS2} P. Aluffi, L. C. Mihalcea, J. Sch\"urmann, C. Su.
in preparation, 2018

\bibitem[BL]{BL}
L. Borisov, A. Libgober: Elliptic genera of singular varieties; Duke Math. J.,  Volume 116, Number 2 (2003), 319--351.

\bibitem[BSY]{BSY}
J.-P. Brasselet, J. Sch{\"u}rmann, and S. Yokura.
\newblock Hirzebruch classes and motivic {C}hern classes for singular spaces.
\newblock {\em J. Topol. Anal.}, 2(1):1--55, 2010.

\bibitem[B]{B} A. Buch, Quiver coefficients of Dynkin type, Michigan Math. J., Vol. 57 (2008), 93--120

\bibitem[BR]{BR}
A. S. Buch, R. Rimanyi: A formula for non-equioriented quiver orbits of type $A$, J. Algebraic Geom. 16 (2007), 531-546

\bibitem[FR1]{tegez}
L. M. Feher, R. Rimanyi: Classes of degeneracy loci for quivers---the Thom polynomial point of view; 
    Duke Math. J., Volume 114, Number 2, August 2002, 193--213 

\bibitem[FR2]{FRcsm}
L. M. Feh\'er, R. Rim\'anyi: Chern-Schwartz-MacPherson classes of degeneracy loci, to appear in Geometry and Topology, {\tt arXiv:1706.05753}

\bibitem[FRW]{FRW}
L. M. Feh\'er, R. Rim\'anyi, A. Weber: Motivic Chern classes and K-theoretic stable envelopes; arXiv 1802.01503, 2018 

\bibitem[FRV]{FRV}
G. Felder, R. Rimanyi, A. Varchenko, {Elliptic dynamical quantum groups and equivariant elliptic cohomology},
Preprint 2017, arXiv:1702.08060

\bibitem[GRTV]{GRTV} 
V. Gorbounov, R. Rimanyi, V. Tarasov, A. Varchenko,
{Cohomology of the cotangent bundle of a flag variety as a Yangian Bethe
algebra} J. Geom. Phys. {74} (2013), 56--86

\bibitem[Ke]{keller} B. Keller,
On cluster theory and quantum dilogarithm identities, in Representations of algebras
and related topics, EMS Ser. Congr. Rep., 85--116, Eur. Math. Soc., Z\"urich (2011)

\bibitem[Kin]{King} 
A. D. King, Moduli of representations of finite-dimensional algebras, Quart. J. Math. Oxford Ser. (2) 45 (1994), no. 180, 515--530.

\bibitem[Kir]{Ki}
A. Kirillov: Quiver representations and quiver varieties, GSM 174, AMS 2016

\bibitem[KS]{KS}
M. Kontsevich, Y. Soibelman: Cohomological Hall algebra, exponential Hodge structures and motivic Donaldson-Thomas invariants. Commun. Number Theory Phys. 5 (2011), no. 2, 231--352

\bibitem[M]{M}  R. MacPherson: Chern classes for singular algebraic varieties, Ann. of Math. 100
(1974), 421--432

\bibitem[MO]{MO1}
D. Maulik, A. Okounkov, {Quantum Groups and Quantum Cohomology}, preprint (2012), 1--276, {arXiv:1211.1287}

\bibitem[O1]{O1} T. Ohmoto: Equivariant Chern classes of singular algebraic varieties with group
actions, Math. Proc. Cambridge Phil. Soc. 140 (2006), 115--134

\bibitem[O2]{O2} T. Ohmoto: Singularities of maps and characteristic classes, to appear in Proc. of Real and Complex Singularities 2012


\bibitem[Ok]{O} A. Okounkov: Lectures on K-theoretic computations in enumerative geometry; preprint 2015, arXiv:1512.07363

\bibitem[PP]{PP} A. Parusi{\'n}ski and P. Pragacz, Chern-Schwartz-MacPherson classes and the Euler
characteristic of degeneracy loci and special divisors, Jour. Amer. Math. Soc. 8 (1995), no. 4, 793--817

\bibitem[Re]{Re}
Reineke M.: Quivers, Desingularizations and Canonical Bases. In: Joseph A., Melnikov A., Rentschler~R. (eds) Studies in Memory of Issai Schur. Progress in Mathematics, vol 210. Birkh\"auser, Boston, MA

\bibitem[Ri1]{qr}
R. Rimanyi: Quiver polynomials in iterated residue form; Journal of Algebraic Combinatorics, Volume 40, Issue 2 (2014), Page 527--542 

\bibitem[Ri2]{rrcoha}
R. Rimanyi: On the Cohomological Hall Algebra of Dynkin quivers, arXiv:1303.3399 

\bibitem[RTV]{RTV3}
R. Rimanyi, V. Tarasov, A. Varchenko:
Elliptic and K-theoretic stable envelopes and Newton polytopes, preprint,
{\tt arXiv:1705.09344 }

\bibitem[RV]{RV} R. Rimanyi, A. Varchenko:
Equivariant Chern-Schwartz-MacPherson classes in partial flag varieties: interpolation and formulae. In {\em Schubert varieties, equivariant cohomology and characteristic classes---IMPANGA 15}, 225--235,  EMS 2018. 

\bibitem[Sch]{Sch} J. Sch\"urmann,
 A generalized Verdier-type Riemann-Roch theorem for Chern-Schwartz-MacPherson classes, arXiv:math/0202175

\bibitem[W]{W} A.  Weber.  Equivariant  Hirzebruch  class  for  singular  varieties. Selecta  Math.  (N.S.),  22(3): 1413--1454, 2016.

\bibitem[YZ]{YZ} Y. Yang, G. Zhao: The cohomological Hall algebra of a preprojective algebra, arXiv:1407.7994






\end{thebibliography}
\end{document}